%% file: formal-desings.tex
\documentclass[a4paper]{amsart}

\usepackage{amssymb, amsmath, amsfonts, stmaryrd}
\usepackage[dvips]{graphicx}
\usepackage[dvips]{color}
\usepackage[all]{xy}
\usepackage[noend]{algorithmic}
\usepackage[ruled]{algorithm}
\mathchardef\ordinarycolon\mathcode`\:
\mathcode`\:=\string"8000
\begingroup \catcode`\:=\active
  \gdef:{\mathrel{\mathop\ordinarycolon}}
\endgroup

\newcommand{\mfm}{\mathfrak{m}}
\newcommand{\mfn}{\mathfrak{n}}
\newcommand{\mfo}{\mathfrak{o}}
\newcommand{\mfc}{\mathfrak{c}}

\newcommand{\xu}{\underline{x}}
\newcommand{\cle}{\preccurlyeq}
\newcommand{\cge}{\succcurlyeq}
\newcommand{\cgt}{\succ}

\newcommand{\ic}{\operatorname{IC}}
\newcommand{\PP}{\mathcal{P}}
\renewcommand{\ll}{\llbracket}
\newcommand{\rr}{\rrbracket}
\newcommand{\lp}{\left(\!\left(}
\newcommand{\rp}{\right)\!\right)}
\newcommand{\lt}{\operatorname{it}}

\newcommand{\Z}{\mathbb{Z}}
\newcommand{\Q}{\mathbb{Q}}
\newcommand{\R}{\mathbb{R}}
\newcommand{\E}{\mathbb{E}}
\newcommand{\F}{\mathbb{F}}

\renewcommand{\P}{\mathbb{P}}

\newcommand{\OO}{\mathcal{O}}
\newcommand{\FF}{\mathcal{F}}
\newcommand{\mm}{\mathfrak{m}}
\renewcommand{\SS}{\mathcal{S}}

\newcommand{\into}{\hookrightarrow}

\newcommand{\ord}{\operatorname{ord}}
\newcommand{\spec}{\operatorname{Spec}}
\newcommand{\proj}{\operatorname{Proj}}
\newcommand{\cent}{\operatorname{center}}
\newcommand{\supp}{\operatorname{supp}}
\newcommand{\edge}{\operatorname{edge}}
\newcommand{\disc}{\operatorname{disc}}

\newcommand{\aut}{\operatorname{Aut}}
\newcommand{\lift}[1]{{}^{\uparrow #1}}
\newcommand{\totfrac}{\operatorname{QF}}
\newcommand{\intclos}[1]{\operatorname{IC}(#1)}
\newcommand{\ratto}{\dashrightarrow}
\newcommand{\map}{\textbf{map}}
\newcommand{\return}{\textbf{return}}

\numberwithin{equation}{section}
\theoremstyle{plain}
\newtheorem{thm}{Theorem}[section]
\newtheorem{lem}[thm]{Lemma}
\newtheorem{cor}[thm]{Corollary}
\newtheorem{cnd}[thm]{Condition}
\theoremstyle{definition}
\newtheorem{dfn}[thm]{Definition}
\theoremstyle{remark}
\newtheorem{rem}[thm]{Remark}
\newtheorem{ex}[thm]{Example}
\newcounter{runexhlp}
\newenvironment{exrun}[1]{%
\newcounter{#1}\setcounter{#1}{\value{thm}}%
\begin{ex}\label{#1}}
{%
\end{ex}}
\newenvironment{excont}[1]{%
\setcounter{runexhlp}{\value{thm}}%
\setcounter{thm}{\value{#1}}%
\begin{ex}[continued]}
{%
\setcounter{thm}{\value{runexhlp}}%
\end{ex}}

\begin{document}

% Frontmatter
% -----------
\title[Formal Desingularization of Surfaces]{Formal Desingularization of Surfaces\\-- The Jung Method Revisited --}

\author{Tobias Beck}
\address{Tobias Beck, Symbolic Computation Group, Johann Radon Institute for Computational and Applied Mathematics, Austrian Academy of Sciences, Altenbergerstra{\ss}e 69, A-4020 Linz, Austria}
\email{Tobias.Beck@ricam.oeaw.ac.at}
\thanks{The author was supported by the FWF (Austrian Science Fund) in the frame of the research project SFB F1303 (part of the  Special Research Program ``Numerical and Symbolic Scientific Computing'').}

\date{\today}

\begin{abstract}
In this paper we propose the concept of \emph{formal desingularizations} as a substitute for the resolution of algebraic varieties. Though a usual resolution of algebraic varieties provides more information on the structure of singularities there is evidence that the weaker concept is enough for many computational purposes. We give a detailed study of the Jung method and show how it facilitates an efficient computation of formal desingularizations for projective surfaces over a field of characteristic zero, not necessarily algebraically closed. The paper includes a generalization of Duval's Theorem on rational Puiseux parametrizations to the multivariate case and a detailed description of a system for multivariate algebraic power series computations.
\end{abstract}

\maketitle
\tableofcontents

% Mainmatter
% ----------

%\clearpage\newpage
\section{Introduction}

\input{introduction.tex}

%\clearpage\newpage
\section{Formal Desingularizations}\label{sec.Concept}

\input{definition.tex}

%\clearpage\newpage
\section{The Method of Jung Revisited}\label{sec.Jung}

\input{jung.tex}

%\clearpage\newpage
\section{Computing with Multivariate Algebraic Power Series}\label{sec.Series}

\input{series.tex}

%\clearpage\newpage
\section{Conclusion}\label{sec.Conclusion}

\input{conclusion.tex}

\appendix
%\clearpage\newpage
\section{Some Local Algebra}

\input{local.tex}

% Endmatter
% ---------
\bibliographystyle{amsplain}
\bibliography{bibfile}

\end{document}

%% file: introduction.tex
Smooth varieties are (in general) well-understood. By contrast (or simply because of that) the objects of interest are often singular varieties. From the theoretical point of view, a remedy for this situation is the celebrated Theorem of Hironaka \cite{HH:1964} (or \cite{HH:2003} for a more modern treatment) on the resolution of singularities which is ubiquitous in algebraic geometry: If $X$ is a variety over a field of characteristic zero, then there always exists a smooth variety $Y$ and a proper birational morphism $\pi: Y \to X$. So for proving theorems and defining birational invariants, one can often argue on $Y$ rather than on $X$ and finally transfer the result back to the singular variety. This theorem has been made constructive by Villamayor \cite{OV:1989}, Bierstone-Milman \cite{EB_PM:1991} and others. There are also two implementations of the desingularization algorithm in \texttt{Singular} \cite{GMG_GP:2002}, one by Bodn{\'a}r and Schicho \cite{GB_JS:2000} and another one by Fr{\"u}bis-Kr{\"u}ger \cite{FK_GP:2004}. In principal this makes many theoretical results algorithmic, but any algorithm relying on a desingularization suffers from the high computational complexity of the desingularization process. There are also specialized constructive methods for the surface case (for a list see \cite{VC_JG_UO:1984} and \cite[Chp.\ 2]{JK:2007}), in particular, the Method of Jung which originates in \cite{HWEJ:1908} and has been further developed in \cite{RW:1935} and \cite{FH:1953}. But -- to the best of our knowledge -- there are no implementations available.

From the computational point of view, it is not always necessary to describe a desingularization completely. In the case of algebraic curves over a field $\E$, series expansions have proven to be an important algorithmic tool. Here the preimage of the singular locus w.r.t.\ a desingularization is a finite set of points. The idea is to describe the desingularization by power series expansions that determine ``formal neighborhoods'' of these points. If $\E$ has characteristic zero, Puiseux expansions can be used, and the Newton-Puiseux algorithm is implemented in many computer algebra systems including {\tt Magma} \cite{MAGMA}, \texttt{Maple} and \texttt{Singular}. The latter system also has an implementation of Hamburger-Noether expansions \cite{AC_JIF:2002} that provides a similar tool for positive characteristic. Applications are for example the computation of an integral basis of the function field \cite{MvH:1994} or Riemann-Roch spaces of divisors \cite{GH_DLB:1995}.

The purpose of this paper is to provide a similar tool for hypersurfaces of $\P_{\E}^3$, where $\E$ is a field of characteristic zero. We emphasize algorithmic aspects and proceed as follows:
In Section~\ref{sec.Concept} we define formal desingularizations for schemes of arbitrary dimension. They can be interpreted as sufficiently large sets of local parametrizations by formal maps. Formal desingularization offer a lot of flexibility because during computation one can always to switch to formally isomorphic schemes.
Then in Section~\ref{sec.Jung} we show how to compute them for surfaces using the method of Jung that depends crucially on the Theorem of Jung-Abhyankar. We define and use rational Puiseux parametrizations whose existence and computability we assume for that moment. We also give a description of the algorithm in mathematical pseudo-code.
Finally Section~\ref{sec.Series} shows in detail how to implement a system that represents and computes with multivariate algebraic power series in {\tt Magma}. (Locally smooth systems, as proposed in \cite{MA_TM_MR:1992}, were insufficient from the complexity point of view.) Folklore knew that the concept of rational Puiseux parametrizations introduced by Duval \cite{DD:1989} should be extensible to multivariate quasi-ordinary polynomials. We give a new and more elementary proof of that fact and show how to compute parametrizations using our representation.
We end with an open problem and an outlook in Section~\ref{sec.Conclusion}. In a short appendix we collect results from local commutative algebra for reference.

Before we proceed we recall and fix some notions. Let $\E$ be a field of characteristic zero and $X$ and $Y$ \emph{integral $\E$-schemes}. By $\E(X)$ and $\E(Y)$ we denote the respective function fields. A \emph{rational map} $\pi: Y \ratto X$ is given by a tuple $(V,\pi)$ s.t.\ $V \subseteq Y$ is open and $\pi: V \to X$ is a regular morphism. Note that we do not restrict to schemes of finite type here. In particular all regular morphisms are rational maps. Two tuples $(V_1,\pi_{1})$ and $(V_2,\pi_{2})$ are \emph{equivalent}, or define the same rational map, if $\pi_{1}|_{V_1 \cap V_2} = \pi_{2}|_{V_1 \cap V_2}$.

Assume that two maps send the generic point of $Y$ to $p \in X$ (its image is always defined for rational maps). Then $(V_1,\pi_{1})$ and $(V_2,\pi_{2})$ are equivalent iff the induced inclusions of fields $\OO_{X,p}/\mm_{X,p} \into \E(Y)$ are the same (where $\mm_{X,p} \subset \OO_{X,p}$ is the maximal ideal). In particular if $\pi$ is \emph{dense}, {\it i.e.}, $p$ is the generic point of $X$, we get an inclusion $\E(X) \into \E(Y)$ determining $\pi$.

Note, however, that not all such field inclusions yield rational maps under our assumption since we have not yet restricted to schemes of finite type over $\E$. {\it E.g.}, let $X := \spec \E[x]$, $Y := \spec \E[x]_{\langle x \rangle}$ and $\pi: Y \to X$ be the morphism induced by localization. Then $\pi$ induces an isomorphism of function fields $\E(X) \cong \E(Y)$. Nevertheless $\pi$ has no rational inverse. A rational map with inverse is called \emph{birational} (or also a \emph{birational transformation}).

Further it is easy to see that dense rational maps may be composed. A rational map has a \emph{domain of definition}, which is the maximal open set on which it can be defined (equivalently, the union of all such open sets).

%%% Local Variables: 
%%% TeX-master: "~/doc/work/academic/worked_texts/Formal-Desing/formal-desings.tex"
%%% End: 

%% file: definition.tex
For this section we denote by $X$ and $Y$ \emph{separated, integral schemes of finite type over $\E$}. Further we assume that both are of \emph{dimension $n$}. All (rational) maps will be relative over $\E$.

Let $(A,\mathfrak{m})$ be a valuation ring of $\E(X)$ over $\E$ (where $\mathfrak{m}$ is the maximal ideal). If $A$ is discrete of rank $1$ and the transcendence degree of $A/\mathfrak{m}$ over $\E$ is $n-1$ then it is called a \emph{divisorial valuation ring of $\E(X)$ over $\E$} or a \emph{prime divisor of $\E(X)$} (see, {\it e.g.}, \cite[Def.~2.6]{MS:1990}). It is an \emph{essentially finite}, regular, local $\E$-algebra of Krull-dimension $1$ ({\it i.e.}, the localization of a finitely generated $\E$-algebra at a prime ideal, see \cite[Thm.~VI.14.31]{OZ_PS:1960}).

Let $(A,\mathfrak{m})$ be a divisorial valuation ring of $\E(X)$ over $\E$. By \cite[Lem.~II.4.4.]{RH:1977} the inclusion $A \subset \E(X)$ defines a unique morphism $\spec \totfrac(A) \to X$ and therefore a rational map $\spec A \ratto X$ sending generic point to generic point. Composing this with the morphism obtained by the $\mathfrak{m}$-adic completion $A \to \widehat{A}$ we get a rational map $\spec \widehat{A} \ratto X$.

\begin{dfn}[Formal Prime Divisor]\label{dfn.FormalDiv}
Let $(A,\mathfrak{m})$ be a divisorial valuation ring of $\E(X)$ over $\E$. Assume that the rational map $\spec \widehat{A} \ratto X$ (as above) is actually a \emph{morphism} $\varphi: \spec \widehat{A} \to X$ ({\it i.e.}, defined also at the closed point). Then $\varphi$ is a representative for a class of schemes up to $X$-isomorphism. This class (and, by abuse of notation, any representative) will be called a \emph{formal prime divisor} on $X$.
\end{dfn}

Hence we may compose a representative $\varphi$ with an isomorphism $\spec B \to \spec \widehat{A}$ to get another representative for the same formal prime divisor. We have an isomorphism $\widehat{A} \cong \F_\varphi\ll t\rr$ with $\F_\varphi := \widehat{A}/\mathfrak{m}\widehat{A} \cong A / \mathfrak{m}$. In the sequel we will sometimes assume that $\widehat{A}$ is already of this form, {\it i.e.}, $\varphi: \spec \F_\varphi\ll t\rr \to X$. The isomorphism is an instance of \emph{Cohen's Structure Theorem} for regular rings, see Theorem~\ref{thm.Cohen}.

Formal prime divisors provide an algorithmic way for dealing with divisorial valuations; A formal prime divisor yields an inclusion of function fields $\E(X) \into \F_\varphi\lp t \rp$. Vice versa, by what was said above, $\varphi$ is determined by this inclusion. So it is this piece of information one has to represent. Composing this inclusion with the order function $\ord_t: \F_\varphi\lp t \rp \to \Z$ we get the corresponding divisorial valuation.

We want to single out a special class of formal prime divisors.

\begin{dfn}[Realized Formal Prime Divisors]\label{dfn.RealizedDiv}
Let $p \in X$ be a regular point of codimension $1$. The formal prime divisor $$\spec \widehat{\OO_{X,p}} \to X$$ (given by composing the canonic morphism $\spec \OO_{X,p} \to X$ with the morphism induced by the completion $\OO_{X,p} \to \widehat{\OO_{X,p}}$) is called \emph{realized}.
\end{dfn}

If $X$ is \emph{normal} then all generic points of closed subsets of codimension $1$ are necessarily regular \cite[Thm~II.8.22A]{RH:1977}. Therefore there is a one-one correspondence of realized formal prime divisors and prime Weil divisors. Another important fact is that we can compare formal prime divisors of birationally equivalent schemes under certain conditions.

\begin{lem}[Pullback of Formal Prime Divisors]\label{lem.DivPullback}
Let $X$ and $Y$ be $S$-schemes with structure morphisms $\rho_X: X \to S$, $\rho_Y: Y \to S$ and assume that $\rho_Y$ is \emph{proper}. Let $\pi: Y \ratto X$ be a \emph{birational transformation} of $S$-schemes (meaning that $\rho_Y = \rho_X  \pi$ as rational maps).

Then each formal prime divisor $\varphi: \spec \widehat{A} \to X$ lifts uniquely to a formal prime divisor on $Y$, {\it i.e.}, there is a unique formal prime divisor $\pi^* \varphi: \spec \widehat{A} \to Y$ s.t.\ $\pi  (\pi^* \varphi) = \varphi$ as rational maps.
\end{lem}
\begin{proof}
Consider the commuting diagram\\
\centerline{\xymatrix{%
\spec \totfrac(\widehat{A}) \ar[r] \ar[d] & \spec \E(X) \ar[rr] \ar[d] & & \spec \E(Y) \ar[d] \\
\spec \widehat{A} \ar[r]^{\varphi} & X \ar[dr]_{\rho_X} \ar@{-->}[rr]^{\pi^{-1}} & & Y \ar[dl]^{\rho_Y} \\
& & S &
}}
where the vertical arrows are given by restricting identity maps to the germs at the generic points and the upper arrows are induced by $\varphi$ and $\pi^{-1}$ respectively. The two squares trivially commute and the triangle commutes because $\pi$ was assumed to be a birational transformation of $S$-schemes.

Now assume that a lift $\pi^* \varphi$ as in the claim exists, then we must have $\pi^* \varphi = \pi^{-1}  \varphi$ because $\pi$ is birational. This (a priori only \emph{rational}) map would fit into the following contracted diagram:\\
\centerline{\xymatrix{%
\spec \totfrac(\widehat{A}) \ar[r] \ar[d] & Y \ar[d]^{\rho_Y} \\
\spec \widehat{A} \ar[r]^{\rho_X  \varphi} \ar@{-->}[ur]^{\pi^* \varphi} & S
}}
Now there exists a unique \emph{regular} morphism $\psi: \spec \widehat{A} \to Y$ which fits into this diagram by applying the valuative criterion of properness to $\rho_Y$ (see \cite[Thm.~II.4.7]{RH:1977}). Since a rational map is uniquely determined by the inclusion of function fields we see that $\pi^* \varphi = \psi$.
\end{proof}

\begin{cor}[Pullback along Proper Morphisms]\label{cor.DivPullback}
Let $\pi: Y \to X$ be a proper, birational morphism. A formal prime divisor on $X$ lifts to a unique formal prime divisor on $Y$. Vice versa a formal prime divisor on $Y$ extends to a unique formal prime divisor on $X$, hence $\pi^*$ is a bijection.
\end{cor}
\begin{proof}
This is obtained by applying Lemma~\ref{lem.DivPullback} to $\pi$ and $\pi^{-1}$ where $S := X$.
\end{proof}

We will apply the operator $\pi^*$ also to sets of formal prime divisors.

\begin{dfn}[Center and Support]\label{dfn.CenterSupp}
Let $\varphi: \spec \widehat{A} \to X$ be a formal prime divisor. We define its \emph{center}, in symbols $\cent(\varphi)$, to be the image of the closed point. Further the \emph{support} of a \emph{finite} set of formal prime divisors $\SS$ is defined as $\supp_X(\SS) := \overline{\{\cent(\varphi) \mid \varphi \in \SS \}}$, {\it i.e.}, the closure of the set of all centers.
\end{dfn}

\begin{ex}
Let $X := \spec A$ with $A := \Q[x,y,z]/\langle x^2+y^2-z^2 \rangle$ be the cone over the circle and $Y := \spec B$ with $B := \Q[x',y',z']/\langle x'^2+y'^2-1 \rangle$ the cylinder. The strict transform $\pi$ under the blow up of the origin has an affine chart $Y \to X$ given by the homomorphism $A \to B: x \mapsto x'z', y \mapsto y'z', z \mapsto z'$. The generic point of the exceptional divisor in $Y$ is the prime ideal $\langle z' \rangle$. In this case we have a trivial isomorphism $\widehat{B_{\langle z' \rangle}} \cong \totfrac(\Q[x',y']/\langle x'^2+y'^2-1 \rangle)\ll z' \rr$. Then a formal prime divisor is induced by the homomorphism \[A \to \totfrac(\Q[x',y']/\langle x'^2+y'^2-1 \rangle)\ll z' \rr: x \mapsto x'z', y \mapsto y'z', z \mapsto z'.\] We can compose this homomorphism with an arbitrary isomorphism of rings to get another representative of the same formal prime divisor. {\it E.g.}, since the circle is a rational curve we can change the coefficient field by $\totfrac(\Q[x',y']/\langle x'^2+y'^2-1 \rangle) \to \Q(s): x' \mapsto \tfrac{2s}{1+s^2}, y' \mapsto \tfrac{-1+s^2}{1+s^2}$ and map, say, $z' \mapsto t + t^2 + \dots$. Now a formal prime divisor $\varphi: \spec \Q(s)\ll t \rr \to X$ is induced by \[A \to \Q(s)\ll t \rr: x \mapsto \frac{2s(t + t^2 + \dots)}{1+s^2}, y \mapsto \frac{(-1+s^2)(t + t^2 + \dots)}{1+s^2}, z \mapsto t + t^2 + \dots.\] One finds $\cent_X(\varphi) = \langle x,y,z \rangle$ (which is the preimage of the prime ideal $\langle t \rangle$) and by construction we know that $\pi^*\varphi$ is realized with center $\langle z' \rangle$ in $Y$.
\end{ex}

Now we are in the situation to define formal desingularizations.

\begin{dfn}[Formal Description of a Desingularization]\label{dfn.FormalDesc}
Let $\pi: Y \to X$ be a \emph{desingularization}, {\it i.e.}, $\pi$ is proper, birational and $Y$ is regular. Let $\SS$ be a \emph{finite} set of formal prime divisors on $X$. We say that $\SS$ is a \emph{formal description} of $\pi$ iff
\begin{enumerate}
\item all divisors in $\pi^* \SS$ are realized,
\item $\pi^{-1}(\supp_X(\SS)) = \supp_Y(\pi^* \SS)$ and
\item the induced morphism $Y \setminus \supp_Y(\pi^* \SS) \to X \setminus \supp_X(\SS)$ is an isomorphism.
\end{enumerate}
\end{dfn}

The set $\SS$ itself consists of formal prime divisors on $X$ and makes no reference to the morphism $\pi$. By another definition we can \emph{avoid} mentioning an \emph{explicit} $\pi$.

\begin{dfn}[Formal Desingularization]\label{dfn.FormalDesing}
Let $\SS$ be a finite set of formal prime divisors on $X$. Then $\SS$ is called a \emph{formal desingularization} of $X$ iff there \emph{exists some} desingularization $\pi$ s.t.\ $\SS$ is a formal description of it.
\end{dfn}

Informally speaking the set $\SS$ makes it possible to treat divisors on $Y$ effectively, although we haven't explicitly represented $Y$ as a whole; Indeed, realized formal prime divisors correspond bijectively to usual prime divisors on the regular scheme $Y$. The set of formal prime divisors now divides into two classes: those within $\pi^* \SS$ and those with center in $Y \setminus \supp_Y(\pi^* \SS)$. The latter can be dealt with on the isomorphic scheme $X \setminus \supp_X(\SS)$. Therefore formal descriptions are an appropriate algorithmic tool to work with invertible sheaves on $Y$.

In the case of surfaces it is easy to see that the existentially quantified $\pi$ in the above definition is actually unique up to isomorphism. Therefore $\SS$ really identifies a desingularization in the common sense. Vice versa, every desingularization can be described formally by completing the stalks along the exceptional divisors.

\begin{thm}[Uniqueness of Surface Desingularization]\label{thm.UniqueSurfDesing}
Let $\SS$ be a formal desingularization of $X$. If $\pi_1: Y_1 \to X$ and $\pi_2: Y_2 \to X$ are two desingularizations described by $\SS$ (in the sense of Definition~\ref{dfn.FormalDesc}) then $Y_1$ and $Y_2$ are isomorphic as $X$-schemes.
\end{thm}
\begin{proof}%[Sketch of proof.]
Let $b^{(0)}: X^{(0)} \to X$ be a minimal desingularization (see \cite[Cor.\ 27.3]{JL:1969}) of $X$ and let $\SS^{(0)}$ be obtained from $(b^{(0)})^*\SS$ by subtracting all realized formal prime divisors. Then $\pi_1$ and $\pi_2$ factor through $X^{(0)}$ yielding a commuting diagram\\
\centerline{\xymatrix{%
Y_1 \ar[dr]_{\pi_1^{(0)}} \ar@{-->}[rr]^{\left(\pi_2^{(0)}\right)^{-1}  \pi_1^{(0)}} & & Y_2 \ar[dl]^{\pi_2^{(0)}} \\
& X^{(0)} &
}}
and $\pi_1^{(0)}$ and $\pi_2^{(0)}$ are both described by $\SS^{(0)}$.

Moreover these maps are proper, birational morphisms between regular surfaces, so they factor into a finite sequence of point blow ups (see \cite[Cor.\ V.5.4]{RH:1977} which holds also for the case of a non-closed ground field). The set of possible centers for the first blow up is exactly $\supp_{X^{(0)}}(\SS^{(0)})$. Choose a center, compute the blow up $b^{(1)}: X^{(1)} \to X^{(0)}$ and set $\SS^{(1)}$ to be the set $(b^{(1)})^*\SS^{(0)}$ excluding the unique formal prime divisor that is turned realized and centered along the exceptional divisor. Again $\pi_1^{(0)}$ and $\pi_2^{(0)}$ factor through morphisms $\pi_1^{(1)}: Y_1 \to X^{(1)}$ and $\pi_2^{(1)}: Y_2 \to X^{(1)}$ described by $\SS^{(1)}$. Going on like this and setting $l := \#\SS^{(0)}$, we find that $\pi_1^{(l)}$ and $\pi_2^{(l)}$ are described by $\SS^{(l)} = \emptyset$ and hence are isomorphisms by Definition~\ref{dfn.FormalDesc}.
\end{proof}

It is not so clear whether a similar statement holds in higher dimensions when minimal resolutions are not available.

\begin{rem}[Formal Desingularization of Reduced Schemes]\label{rem.DesingReduced}
The above definitions can and will be used in a more general setting. Namely, if $X = \bigcup_i X_i$ is the decomposition of a \emph{reduced, equidimensional} (no longer integral) scheme into irreducible closed subschemes then any morphism $\spec \F\ll t\rr \to X$ is actually a morphism to one of the $X_i$. We call it a formal prime divisor if the corresponding $\spec \F\ll t\rr \to X_i$ is a formal prime divisor. We call it again realized iff it corresponds to the completion at the germ of a regular codimension $1$ point in $X$ (not in $X_i$!). Lemma~\ref{lem.DivPullback} and Corollary~\ref{cor.DivPullback} remain valid in this setting (where a birational morphism between two reduced schemes is a morphism that induces birational morphisms on all irreducible components). The definitions of center, support and formal desingularization carry over straight forward.
\end{rem}

%%% Local Variables: 
%%% TeX-master: "~/doc/work/academic/worked_texts/Formal-Desing/formal-desings.tex"
%%% End: 

%% file: jung.tex
In this section we describe the desingularization of surfaces after Jung.

\subsection{Theory of the Method}

First we will view a projective surface as a certain covering of a smooth surface. Then we modify the covering such that after passing to the integral closure the remaining singularities are very simple and can be resolved by point blow ups.

\subsubsection{Projective Surfaces as Ramified Coverings of the Plane}\label{sec.RamCover}

Consider a projective hypersurface $X \subseteq \P_{\E}^3$. We want to view $X$ in a slightly different way. Let $p_0 \in \P_{\E}^3 \setminus X$ be a closed \emph{rational} point ({\it i.e.}, its residue field is isomorphic to $\E$) and $\pi: W \to V$ the linear projection from $p_0$ where $W :=  \P_{\E}^3 \setminus p_0$ and $V := \P_{\E}^2$. This projection defines a line bundle. Its restriction $\pi\mid_X: X \to V$ is a Noether normalization of $X$, {\it i.e.}, a \emph{finite morphism} onto the projective plane. We subsume the governing properties in the following definition:

\begin{dfn}[Ramified Coverings]
Let $\pi: W \to V$ be a \emph{line bundle} s.t.\ $V$ is a \emph{regular}, \emph{integral} surface over $\E$. Further let $X \subset W$ be a \emph{reduced hypersurface} s.t.\ $\pi\vert_X: X \to V$ is \emph{finite}. The tuple $(\pi, X)$ is then called a \emph{ramified covering}.
\end{dfn}

That is our notion of ramified covering comprises that we are dealing with surfaces and that the covering surface is embedded in a line bundle over the base. Now we want to define the ramification locus of such a covering, {\it i.e.}, the locus where the covering $\pi\vert_X$ is ``not locally trivial'' (more precisely, not {\'e}tale). Since $\pi$ is a line bundle and $\pi\vert_X$ is finite, we can find a covering $\{U_i\}_i$ of $V$ by affine open subsets s.t.\ $\OO_W(\pi^{-1}(U_i)) \cong \OO_V(U_i)[z]$ is a polynomial ring in one variable and $X$ is given by a monic, squarefree polynomial $f_i \in \OO_V(U_i)[z]$.

\begin{dfn}[Discriminant Curve]
Let $(\pi: W \to V, X)$ be a ramified covering. Let $\{U_i\}_i$ be a covering of $V$ as above and $f_i \in \OO_V(U_i)[z]$ a set of polynomials defining $X$. The ideals $\langle \disc_z(f_i) \rangle \subseteq \OO_V(U_i)$ define an invertible sheaf of ideals and the corresponding subscheme is called the \emph{discriminant curve} $D_{\pi\vert_X} \subseteq V$.
\end{dfn}

The discriminant curve is actually a concept independent of the concrete covering and embedding of $X$ into $W$. It depends only on $\pi\vert_X$. Further the covering is locally trivial except over $D_{\pi\vert_X}$.

\begin{figure}
\centerline{\hfill
\fbox{\includegraphics[width=5cm]{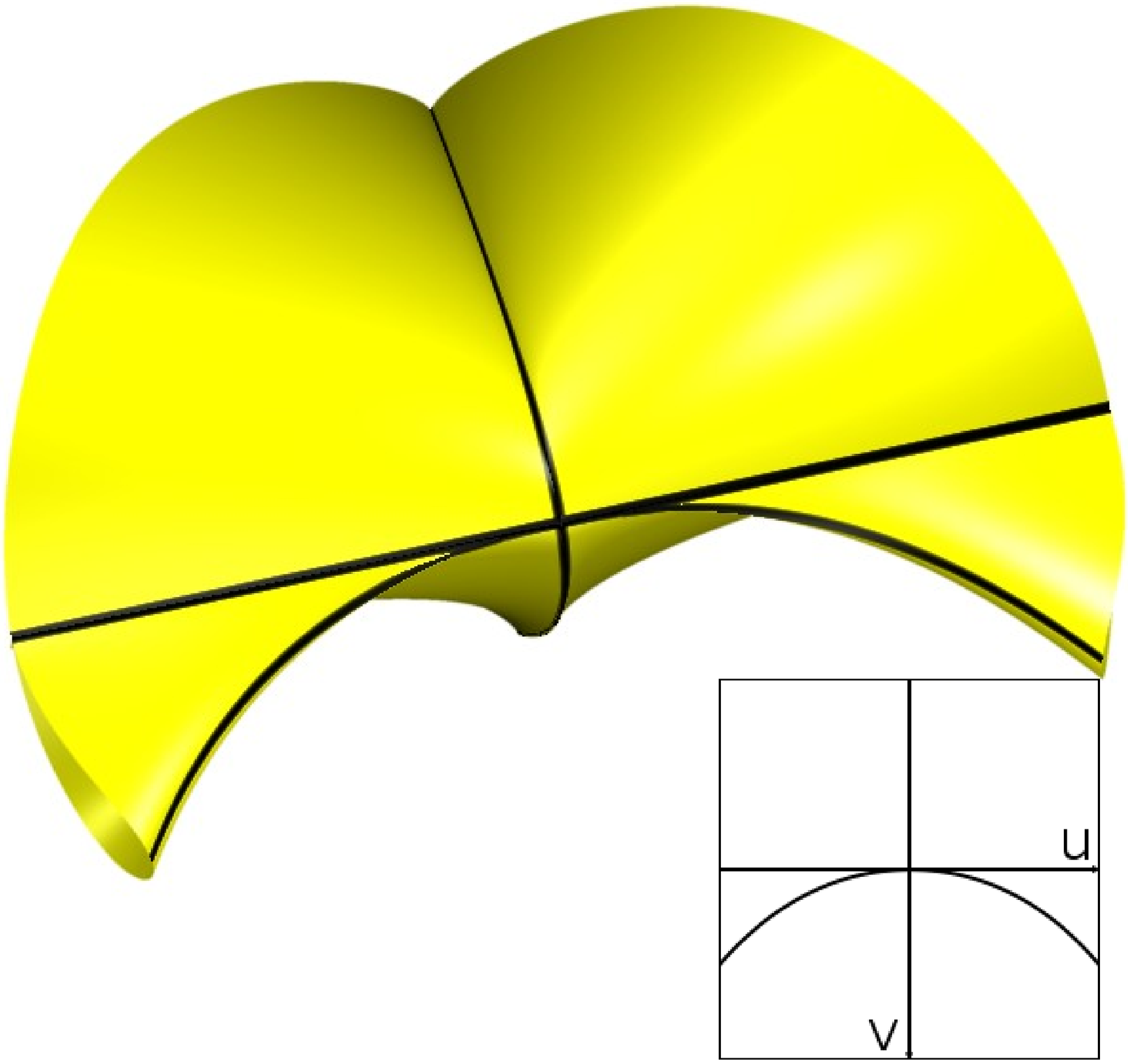}}\hfill
\fbox{\includegraphics[width=5cm]{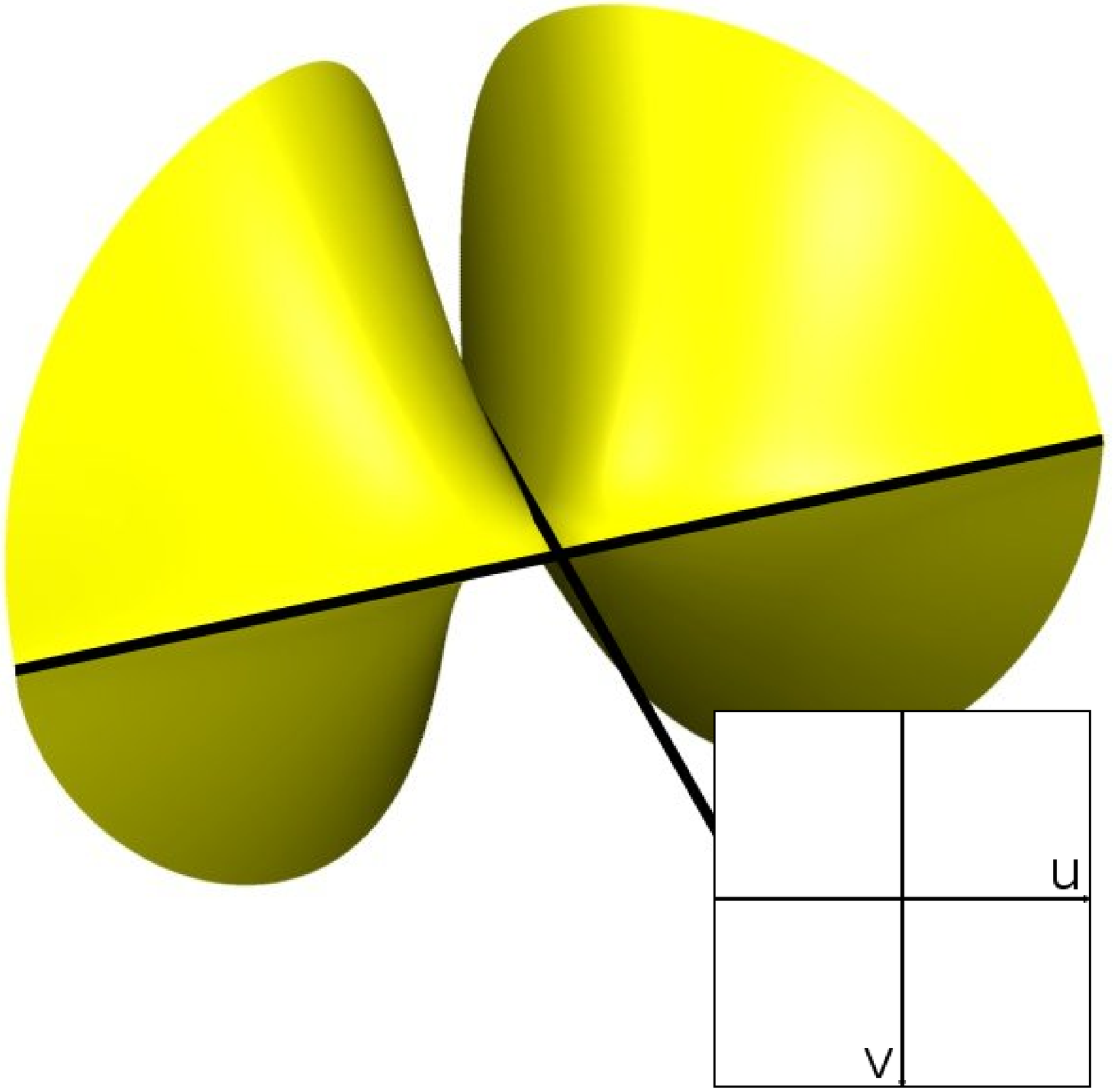}}\hfill
}
\caption{Embedded Desingularization of the Discriminant}\label{fig.EmbDiscDesing}\mbox{}\\[-1ex]
\parbox{\textwidth}{\tiny
The left side shows the local picture of the example surface $X$ together with its discriminant curve. On the right we depict a chart $X_0$ obtained from the embedded desingularization of the discriminant curve.
}
\end{figure}

\begin{exrun}{run.FormalDesing}
Consider the surface $X \subset \P_{\Q}^3$ given by $F = 0$ with $F := x_0^6 + 3x_0^4x_2x_3 + x_0^3x_1^2x_2 + 3x_0^2x_2^2x_3^2 + x_2^3x_3^3 \in \Q[x_0,x_1,x_2,x_3]$. Since $F$ is monic in $x_0$ the surface $X$ doesn't contain the point $p := (1:0:0:0)$. The line bundle defined by the projection of $X \setminus \{p\}$ to the plane $x_0 = 0$ is given by dehomogenizing with respect to $x_1$, $x_2$ and $x_3$.

In the last chart, the defining equation has the form $f := w^6 + 3vw^4 + u^2vw^3 + 3v^2w^2 + v^3 \in \Q[u,v][w]$ where we have mapped $x_0 \mapsto w$, $x_1 \mapsto u$, $x_2 \mapsto v$ and $x_3 \mapsto 1$. We have the local discriminant $\disc_w(f) = 729 u^8v^{12}(u^4 - 64v)$. Figure~\ref{fig.EmbDiscDesing} left displays the surface.
\end{exrun}

\subsubsection{Embedded Desingularization of the Discriminant Curve}\label{sec.EmbDesing}

We would like to give a more detailed study of the covering, but for general ramified coverings $(\pi, X)$ this is hard. The complexity of $X$ resp.\ of the covering map is partially reflected in the discriminant curve $D_{\pi\vert_X}$. Recall that a closed point $p$ is called \emph{normal crossing} for the embedded curve $D_{\pi\vert_X} \subseteq V$ if the curve is locally defined by $d_x^{e_x} d_y^{e_y} u \in \OO_{V,p}$ where $u$ is a unit, $\{d_x,d_y\}$ is a local system of parameters and $e_x,e_y \ge 0$. The whole curve is called normal crossing if it is normal crossing at every closed point.

\begin{dfn}[Nicely Ramified Coverings]
Let $(\pi, X)$ be a ramified covering. If $D_{\pi\vert_X}$ is normal crossing then we call $(\pi, X)$ \emph{nicely ramified}.
\end{dfn}

We can always modify a ramified covering to become nicely ramified. If we speak of \emph{normal crossing singularities} of the discriminant curve of a nicely ramified covering, we mean the closed points where two components intersect.

\begin{lem}[Simplification of Coverings]
Let $(\pi: W \to V, X)$ be a ramified covering. There is a \emph{proper}, \emph{birational} morphism $\rho: V' \to V$ s.t.\ the ramified covering $(\pi': W' \to V', X')$ is \emph{nicely ramified} (where $W' := W \times_V V'$ and $X' := X \times_W W'$). Further $\rho$ is given by a finite succession of blow ups in closed points.
\end{lem}
\begin{proof}
The theorem on embedded desingularization of curves (see \cite[Thm.\ 1.47]{JK:2007}) shows the existence of the morphism $\rho$; Indeed it says that after a finite number of blow ups in the singular points of the reduced curve the pullback of $D_{\pi\vert_X}$ is normal crossing. For showing the lemma it remains to prove that constructing the discriminant curve commutes with base extension. This is left to the reader.
\end{proof}

\begin{excont}{run.FormalDesing}
The curve defined by $\disc_w(f)$ has a complicated singularity at the origin which needs to be resolved. One of the chart maps is given by $u \mapsto uv$, $v \mapsto u^2v^3$ and transforms $\disc_w(f)$ to $729 u^{34}v^{47}(u^2v - 64)$ which describes a curve with a normal crossing intersection at the origin.

The embedded desingularization of the discriminant curve can be applied to the surface $X$ (by mapping $w \mapsto w$) to obtain a surface with chart $X_0$. We obtain the new local equation $f_0 = w^6 + 3u^2v^3w^4 + u^4v^5w^3 + 3u^4v^6w^2 + u^6v^9$ with $\disc_w(f_0) = 729 u^{34}v^{47}(u^2v - 64)$, see Figure~\ref{fig.EmbDiscDesing} right.
\end{excont}

\subsubsection{Desingularization of Toroidal Surface Singularities}\label{sec.ToroidalSing}

The structure of nicely ramified coverings depends crucially on the celebrated Theorem of Jung-Abhyankar (see Theorem~\ref{thm.JungAbhyank}). A polynomial fulfilling the conditions of the theorem is called \emph{quasi-ordinary}. In its original form the theorem is not precise enough for our purposes, for example, because the statement doesn't involve the coefficient fields of the power series solutions. In the sequel $\E\ll\xu^{\Gamma}\rr$ will denote the ring of power series with coefficients in a field $\E$, variables $x_1,\dots,x_n$ and exponents in $\Gamma \cap \R_{\ge 0}^n$ (the non-negative orthant of a full rational lattice). We need the concept of rational parametrizations (for a refined version see Definition~\ref{dfn.RatParFine}):

\begin{dfn}[Parametrizations]\label{dfn.RatParCoarse}
Let $f \in \E\ll x_1,\dots,x_n\rr[z]$ be quasi-ordinary. {\it I.e.}, $f$ is monic, squarefree and s.t.\ $\disc_z(f) = x_1^{e_1} \cdots x_n^{e_n} u(\xu)$ where $u(0,\dots,0) \neq 0$.

We call $(\sigma, \alpha)$ with $\sigma \in \aut(\E'\ll x_1,\dots,x_n\rr \mid \E')$ and $\alpha \in \E'\ll\xu^{\Gamma}\rr$ a \emph{parametrization} of $f$ if $\E \subseteq \E'$, $\Z^n \subseteq \Gamma$ and $\sigma\lift{z}(f)(\alpha)=0$.
Let $g \vert f$ be an irreducible factor s.t.\ $\sigma\lift{z}(g)(\alpha)=0$. We call $(\sigma, \alpha)$ \emph{rational} if the induced homomorphism $\ic(\E\ll x_1,\dots,x_n\rr[z]/\langle g \rangle) \to \E'\ll\xu^{\Gamma}\rr$ which maps $z \mapsto \alpha$ and $\gamma \mapsto \sigma(\gamma)$ for $\gamma \in \E\ll x_1,\dots,x_n\rr$ is an isomorphism.
A set of rational parametrizations for $f$ is called \emph{complete} if it is in bijective correspondence with the irreducible factors of $f$.
\end{dfn}

Here we used $\sigma\lift{z}$ for the lifting of the automorphism to the polynomial ring by coefficient-wise application. We will show later, that we can actually compute such rational parametrizations.

\begin{thm}[Existence of Rational Parametrizations]\label{thm.ExistRatPar}
Let $f \in \E\ll x_1,\dots,x_n\rr [z]$ be a \emph{quasi-ordinary} polynomial. Then a complete set of rational parametrizations of $f$ exists and can be computed. Moreover, if actually $\disc_z(f) = x_1^{e_1} \cdots x_m^{e_m} u(\xu)$ where $m \le n$ and $u(0,\dots,0) \neq 0$ then the exponent lattices of the power series rings will be of the form $\Gamma \times \Z^{n-m}$ for some $m$-dimensional rational lattice $\Gamma$.
\end{thm}
\begin{proof}
For the first statement see Algorithm~\ref{alg.Param} and Corollary~\ref{cor.RatPar} of Section~\ref{sec.AlgSeries}. In Lemma~\ref{lem.Completeness} the relation between the computed parametrizations and the fractionary power series roots of $f$ is explored. Together with Theorem~\ref{thm.JungAbhyank} this gives the statement about the exponent lattice.
\end{proof}

\begin{excont}{run.FormalDesing}
The transformed polynomial $f_0 \in \Q[u,v][w]$ is quasi-ordinary. In this case a complete set of rational parametrizations is given by only a single parametrization $(\sigma, \alpha)$ with $\sigma: \Q\ll u,v\rr \to \Q\ll u,v\rr: u \mapsto -8u, v \mapsto -v$ and \[\alpha := - 8u^{\tfrac{6}{6}}v^{\tfrac{9}{6}} + 8u^{\tfrac{8}{6}}v^{\tfrac{10}{6}} - 4u^{\tfrac{10}{6}}v^{\tfrac{11}{6}} + u^{\tfrac{14}{6}}v^{\tfrac{13}{6}} - \tfrac{1}{2}u^{\tfrac{18}{6}}v^{\tfrac{15}{6}} + \tfrac{5}{16}u^{\tfrac{22}{6}}v^{\tfrac{17}{6}} - \tfrac{7}{32}u^{\tfrac{26}{6}}v^{\tfrac{19}{6}} + \dots\]
We will see later that the simple form of $\sigma$ is no coincidence. We further have $\alpha \in \Q\ll (u,v)^{\Gamma} \rr$ with $\Gamma := \Z (0,\tfrac{1}{2}) + \Z (\tfrac{1}{3}, \tfrac{1}{6})$. This lattice is shown in Figure~\ref{fig.Lattices} below.
\end{excont}

Complete sets of rational parametrizations describe very explicitely the structure of the integral closure of $\E\ll x_1, \dots, x_n \rr[z]/\langle f \rangle$.

\begin{lem}[Decomposition Induced by Rational Parametrizations]\label{lem.RatParDecomp}
Let $f \in \E\ll x_1, \dots, x_n \rr[z]$ be a \emph{quasi-ordinary} polynomial and $\PP = \{(\sigma_i, \alpha_i)\}_{i}$ with $(\sigma_i, \alpha_i) \in \aut(\E_i\ll x_1, \dots, x_n \rr \mid \E_i) \times \E_i\ll\xu^{\Gamma_i}\rr$ a complete set of rational parametrizations of $f$. Then the homomorphisms \[\psi_i: \E\ll x_1, \dots, x_n \rr[z]/\langle f \rangle \to \E_i\ll\xu^{\Gamma_i}\rr: \gamma \mapsto \sigma_i(\gamma) \text{ for } \gamma \in \E\ll x_1, \dots, x_n \rr \text{ and } z \mapsto \alpha_i \] can be composed to a homomorphism \[\psi: \E\ll x_1, \dots, x_n \rr[z]/\langle f \rangle \to \prod_i \E_i\ll\xu^{\Gamma_i}\rr: a \mapsto (\dots,\psi_i(a),\dots)\] which lifts to an isomorphism $\psi: \intclos{\E\ll x_1, \dots, x_n \rr[z]/\langle f \rangle} \to \prod_i \E_i\ll\xu^{\Gamma_i}\rr$.
\end{lem}
\begin{proof}[Sketch of Proof]
By the definition of rational parametrizations it is enough to show
\[\intclos{\E\ll x_1, \dots, x_n \rr[z]/\langle f \rangle} \cong \textstyle\prod_i \intclos{\E\ll x_1, \dots, x_n \rr[z]/\langle f_i \rangle}\] where the righthand side runs over all irreducible factors $f_i$ of $f$. By \cite[Exer.\ 2.26]{DE:1995} this can be shown using orthogonal idempotents. More precisely, if $f = h_1 h_2$ where $h_1$ and $h_2$ are factors without common divisor then set $e_1 := h_2/(h_1+h_2)$ and $e_2 := h_1/(h_1+h_2)$ to be elements in $\totfrac(\E\ll x_1, \dots, x_n \rr[z]/\langle f \rangle)$. (For this one checks that the denominators are no zero-divisors.) Then one easily computes $e_1 + e_2 = 1$, $e_1 e_2 = 0$ and $e_i = e_i(e_1 + e_2) = e_i^2 + e_1e_2 = e_i^2$. The idempotency relations imply in particular that $e_1, e_2 \in \intclos{\E\ll x_1, \dots, x_n \rr[z]/\langle f \rangle}$. So the integral closure splits: \[\intclos{\E\ll x_1, \dots, x_n \rr[z]/\langle f \rangle} \cong e_1 \intclos{\E\ll x_1, \dots, x_n \rr[z]/\langle f \rangle} \times e_2 \intclos{\E\ll x_1, \dots, x_n \rr[z]/\langle f \rangle}\] Finally one shows $e_i \intclos{\E\ll x_1, \dots, x_n \rr[z]/\langle f \rangle} \cong \intclos{\E\ll x_1, \dots, x_n \rr[z]/\langle h_i \rangle}$.
\end{proof}

\begin{cor}[Singularities of Ramified Coverings]\label{cor.ToroidalSing}
Let $(\pi, X)$ be a ramified covering, $\nu: \widetilde{X} \to X$ the \emph{normalization morphism}. Then the (isolated) singular points of $\widetilde{X}$ lie over the singularities of the (reduced) discriminant curve. Moreover, over normal crossing singularities the singularities of $\widetilde{X}$ are \emph{toroidal}.
\end{cor}
\begin{proof}
Let $q \in \widetilde{X}$ be a closed point and $p := \pi(\nu(q))$. Assume that $p$ either does not lie at all on $D_{\pi\vert_X}$, or lies on a regular point of the reduced curve, or is a normal crossing singularity. In each of these cases $D_{\pi\vert_X}$ can locally be defined by $d_x^{e_x}d_y^{e_y} u \in \OO_{V,p}$ where $\{d_x,d_y\}$ is a local system of parameters and $e_x \ge 0$, $e_y \ge 0$. There is an isomorphism between the completion of $\OO_{V,p}$ and $\F\ll x,y\rr$ (where $\F$ is the residue field of $\OO_{V,p}$) which maps $d_x \mapsto x$ and $d_y \mapsto y$, see Theorem~\ref{thm.Cohen}. We can as well assume that $\widehat{\OO_{V,p}} = \F\ll x,y\rr$, $d_x = x$ and $d_y = y$.

The completion of the fiber of $\pi\vert_X$ can be defined by the vanishing of a polynomial $f \in \F\ll x,y\rr[z]$ which is quasi-ordinary. Building the integral closure commutes with completion, see Lemma~\ref{lem.NormCompl}. Then Lemma~\ref{lem.RatParDecomp} implies that the completion of $\OO_{\widetilde{X},q}$ is isomorphic to a power series ring with fractionary exponents. But those are the completions of the distinguished stalks of affine toric surfaces, hence, all such points $q$ can at most be toroidal.

If $p$ does not lie on $D_{\pi\vert_X}$ then $e_x = e_y = 0$. If it is a regular point of $D_{\pi\vert_X}$ then, say, $e_x > 0$ and $e_y = 0$. By Theorem~\ref{thm.ExistRatPar}, in both cases, the power series rings in the parametrizations have exponent lattices of the form $(\tfrac{1}{e} \Z) \times \Z$. These rings are regular and by faithful flatness $\OO_{\widetilde{X},q}$ must be regular itself.
\end{proof}

\begin{excont}{run.FormalDesing}
For our example this means that the integral closure of $\Q\ll u,v\rr[w]/\langle f_0 \rangle$ is isomorphic to $\Q\ll (u,v)^{\Gamma} \rr$. This fits well to the picture in Figure~\ref{fig.EmbDiscDesing} that suggests that the surface $X_0$, though singular, has only a single analytic branch at the origin.
\end{excont}

Modifying a ramified covering to become nicely ramified is constructive, since embedded desingularization of curves is. Passing to the normalization of a scheme of finite type over $\E$ is constructive, since computing the integral closure of finitely generated $\E$-algebras is (see, {\it e.g.}, \cite{TdJ:1998}). By the above corollary we are left with the task of desingularizing normal toroidal surface singularities.

These singularities have first been studied by Jung \cite{HWEJ:1908}. As noted in \cite[Lect.\ 2, \S 2]{VC_JG_UO:1984} a normal toric surface can be desingularized by a finite sequence of point blow ups. This property can be transferred to toroidal singularities by Lemma~\ref{lem.BlowCompl}. So we could obtain a desingularization by computing a sequence of point blow ups of the normalization $\widetilde{X}$ and, hence, Jung's method is already constructive.

However, computing normalizations is not cheap. Lemma~\ref{lem.RatParDecomp} shows that rational parametrizations anyway describe the integral closure. In the following section we will therefore follow a different approach, always with formal desingularizations in mind. We will also benefit in other places from the additional flexibility obtained by applying formal isomorphisms.

The final lemma of this section has its origins in Hirzebruch \cite{FH:1953} (see also \cite{HBL:1971}) who first gave explicit formulas for the desingularization of toric surface singularities using continued fractions. Afterwards the arrival of toric geometry ({\it cf.} \cite{WF:1993}) has introduced new terminology and means of description. Recall that the \emph{dual} $\Gamma^{\vee}$ of a full lattice $\Gamma \subset \Q^n$ is the set of all linear forms $\mfn \in \Q^n$ with $\mfn(\mfm) \in \Z$ for all $\mfm \in \Gamma$.

\begin{figure}
\centerline{\hfill
\input{lattice.030.tex}\hfill
\input{lattice-dual.030.tex}\hfill
}\caption{Lattice and Dual Lattice}\label{fig.Lattices}\mbox{}\\[-1ex]
\parbox{\textwidth}{\tiny
We show the non-negative quadrant of a lattice $\Gamma$ and its dual $\Gamma^\vee$ ({\it i.e.}, the set of all linear forms to $\Z$ identified via scalar multiplication of vectors). In the dual we have inscribed an minimal sequence of generators $\mfn_1,\dots,\mfn_4$. Indices are chosen s.t.\ subsequent pairs correspond to neighboring vectors.
}
\end{figure}
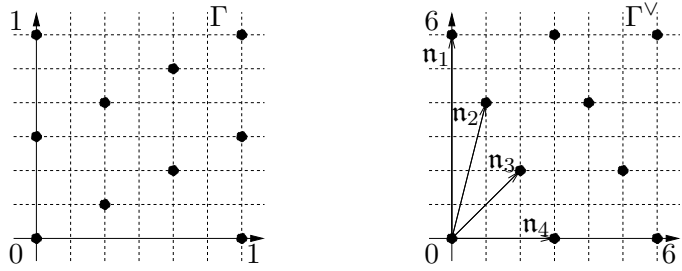

\begin{lem}[Formal Desingularization of Toric Surfaces]\label{lem.ToricSing}
Let $\spec \E[\xu^\Gamma]$ for some rational lattice $\Gamma \subset \Q^2$ be a toric surface and $\mfn_1,\dots,\mfn_l \in \Gamma^{\vee} \cap \R_{\ge 0}^2$ an \emph{ordered and minimal} sequence of monoid generators as in Figure~\ref{fig.Lattices}. Then the set of morphisms $\psi_i: \spec \E(s)\ll t\rr \to \spec \E[\xu^\Gamma]$ given by the $\E$-algebra homomorphisms \[\E[\xu^\Gamma] \to \E(s)\ll t\rr: \xu^{\mfm} \mapsto s^{\mfn_i(\mfm)} t^{\mfn_{i+1}(\mfm)}\] for $1 \le i \le l-2$ is a formal desingularization of $\spec \E[\xu^\Gamma]$.
\end{lem}
\begin{proof}
Desingularizations of toric schemes can be constructed using special fans in the dual cone and lattice $\Gamma^{\vee} \cap \R_{\ge 0}^2$ (see, {\it e.g.}, \cite{WF:1993,DC:2000}). Applying this construction to the fan whose one-dimensional faces are given by the $\mfn_i$ one obtains a desingularization $\pi$. In fact the morphisms $\pi_i: \spec \E[s,t] \to \spec \E[\xu^\Gamma]$ given by $\E[\xu^\Gamma] \to \E[s,t]: \xu^{\mfm} \mapsto s^{\mfn_i(\mfm)} t^{\mfn_{i+1}(\mfm)}$ for $1 \le i \le l-1$ are isomorphic to the restriction of $\pi$ to an open covering by affine charts.

The exceptional locus ({\it i.e.}, the $\pi$-preimage of the isolated singularity) is a finite union of divisors. We get a formal description of $\pi$ by completing the local rings along these divisors. They are given, for example, by the prime ideals $\langle t \rangle \subset \E[s,t]$ in each of the charts $1,\dots,l-2$. For the completions of the local rings we then have trivially $\widehat{\E[s,t]_{\langle t\rangle}} \cong \E(s)\ll t\rr$.
\end{proof}

\begin{rem}[Minimal Desingularization by Point Blow Ups]\label{rem.MinimalDesing}
The desingularization described by the above set of morphisms is the \emph{minimal} one obtained (up to isomorphism) by subsequently blowing up isolated singular points. To show this one observes that every single such blow up gives (non-affine) toric surfaces that are described by fans in $\Gamma^{\vee} \cap \R_{\ge 0}^2$. Elementary arguments about the lattice and its dual show that the one-dimensional faces of these lattices are always given by one of the vectors $\mfn_i$. On the other hand any fan defined by a proper subset of the generators is associated to a still singular surface.
\end{rem}

\begin{excont}{run.FormalDesing}
This lemma can be applied as follows. The normalization $\widetilde{X_0}$ of $X_0$ has an isolated singularity which is formally isomorphic to the distinguished germ of a toric surface with coordinate ring $\Q[(u,v)^{\Gamma}]$.

The lemma says that this toric surface is formally desingularized by mapping $(u,v)^{\mfm} \mapsto s^{\mfn_{i}(\mfm)}t^{\mfn_{i+1}(\mfm)} \in \Q(s)\ll t\rr$ for $i \in \{1,2\}$ with $\mfn_i$ as in Figure~\ref{fig.Lattices}. Let's look at this for $i=1$ in terms of algebra generators $v^{1/2}, u^{1/3}v^{1/6}, u \in \Q[(u,v)^{\Gamma}]$. We have to map
\begin{align*}
v^{\tfrac{1}{2}} & \mapsto s^{\langle (0,6), (0,1/2) \rangle} t^{\langle (1,4), (0,1/2) \rangle} = s^3 t^2, \\
u^{\tfrac{1}{3}}v^{\tfrac{1}{6}} & \mapsto s^{\langle (0,6), (1/3,1/6) \rangle} t^{\langle (1,4), (1/3,1/6) \rangle} = s t, \\
u & \mapsto s^{\langle (0,6), (1,0) \rangle} t^{\langle (1,4), (1,0) \rangle} = t.
\end{align*}

Composing with the homomorphism induced by the rational parametrization $(\sigma, \alpha)$ we get homomorphisms $\varphi_i: \Q[u,v][w]/\langle f_0 \rangle \to \Q(s)\ll t \rr$:
\begin{align*}
\varphi_1: & ~u \mapsto -8s^{6}t, v \mapsto -t^{4},
w \mapsto - 8s^6t^7 + 8s^8t^8 - 4s^{10}t^9 + s^{14}t^{11} - \tfrac{1}{2}s^{18}t^{13} + \dots, \\
\varphi_2: & ~u \mapsto -8st^2, v \mapsto -s^4t^2,
w \mapsto - 8s^7t^5 + 8s^8t^6 - 4s^9t^7 + s^{11}t^9 - \tfrac{1}{2}s^{13}t^{11} + \dots
\end{align*}
The two induced morphisms $\spec \Q(s)\ll t \rr \to X_0$ are the formal prime divisors centered at the origin that become realized on a desingularization of $\widetilde{X_0}$ by point blow ups.

But this is not yet a formal desingularization of $X_0$. As can be seen in Figure~\ref{fig.EmbDiscDesing}, the singular locus of $X_0$ is the union of the $u$-axis and the $v$-axis. So there should also be formal prime divisors supported on these lines. Let's first consider $u=0$. By mapping $u \mapsto t$ and $v \mapsto s$ we transform $f_0$ to $w^6 + 3s^3t^2w^4 + s^5t^4w^3 + 3s^6t^4w^2 + s^9t^6 \in \Q(s)[t][w]$. Now we compute a univariate rational parametrization $(\sigma, \alpha)$: \[\sigma: t \mapsto \tfrac{-8}{s^5} t,~ \alpha = - \tfrac{8}{s^5} \gamma t^{\tfrac{1}{2}} - \tfrac{8}{s^5} t^{\tfrac{2}{3}} + \tfrac{4}{s^8} \gamma t^{\tfrac{5}{6}} + \tfrac{1}{s^{11}} \gamma t^{\tfrac{7}{6}} + \tfrac{1}{2s^{14}} \gamma t^{\tfrac{3}{2}} + \dots \in \Q(s)[\gamma]\ll t^{(1/6)\Z} \rr\] Here $\gamma$ is an algebraic element satisfying $\gamma^2 + s^3 = 0$. Canceling exponent denominators this yields a homomorphism $\varphi_0: \Q[u,v][w]/\langle f_0 \rangle \to \Q(s)[\gamma]\ll t \rr$:
\[
\varphi_0: u \mapsto \tfrac{-8}{s^5} t^6, v \mapsto s, w \mapsto - \tfrac{8}{s^5} \gamma t^3 - \tfrac{8}{s^5} t^4 + \tfrac{4}{s^8} \gamma t^5 + \tfrac{1}{s^{11}} \gamma t^7 + \tfrac{1}{2s^{14}} \gamma t^9 + \dots
\]
This homomorphism corresponds to completing the germ at a generic point in $\widetilde{X_0}$ above $u=0$ by Lemma~\ref{lem.NormCompl}. This germ is not modified by subsequent point blow ups. So we get a further formal prime divisor. A last one is supported on $v=0$:
\[
\varphi_3: u \mapsto s, v \mapsto -64s^{10}t^6, w \mapsto - 512s^{16}t^{9} + 512s^{18}t^{10} - 256s^{20}t^{11} + 64s^{24}t^{13} + \dots
\]

Note that this procedure explicitly produces the residue fields at the generic points of the exceptional divisors in a desingularization. They correspond to the coefficient fields of the power series. Finally, composing all four formal prime divisors with the morphism $X_0 \to X$ which was obtained by the desingularization of the discriminant curve, we get a part of a formal desingularization of $X$.
\end{excont}

\subsection{A Divide and Conquer Approach}

Now we want to cast the theory of the previous paragraph into explicit algorithms. We want to give a formal description of a Jung desingularization of a ramified covering $(\pi: W \to V, X)$. Note that such a desingularization is always relative to an embedded desingularization $\rho: V_\rho \to V$ of the discriminant curve $D_{\pi\vert_X}$. Because then $(\pi_\rho: W_\rho \to V_\rho, X_\rho)$ for $W_\rho := W \times_V V_\rho$ and $X_\rho := X \times_W W_\rho$ is a nicely ramified covering and we can define the Jung desingularization $\Pi_\rho: Y_\rho \to X$ to be obtained by normalizing $X_\rho$ and successively blowing up singular points, see Remark~\ref{rem.MinimalDesing}.

If we wanted to avoid redundant blow ups, we could already fix the embedded desingularization $\rho$ to be minimal, {\it i.e.}, obtained by blowing up a point of the discriminant curve only when it is a non-normal crossing singularity. For computational reasons, we choose a slightly different desingularization, see Remark~\ref{rem.AlmostMinimal} below.
Also, the definition of a formal description $\SS$ of $\Pi_\rho$ leaves a bit of choice for $\SS$, {\it e.g.}, one may always add formal prime divisors which are realized on the desingularization. We get rid of this ambiguity by requiring that $\supp_X(\SS) = \pi^{-1}(D_{\pi\vert_X})$ and say we compute a \emph{formal description of $\Pi_{\rho}$ above $D_{\pi\vert_X}$}.

\begin{rem}[Divide and Conquer Paradigm]\label{rem.DivAndConqPara}
We will desingularize the discriminant curve and compute the formal prime divisors for the surface at the same time. The substitutions involved in computing the curve desingularization make the defining equations more complicated. Therefore our paradigm must be to compute formal prime divisors ``as early'' as possible. In other words, if we know that the surface (resp.\ its normalization) remains unchanged in a subset (up to isomorphism) by further blow ups of the discriminant curve we immediately compute the formal prime divisors centered in that set.

Let $\rho: V_\rho \to V$ be an embedded desingularization of the discriminant curve. Our divide and conquer approach (in particular Algorithm~\ref{alg.DesingRecursive} below) relies on the following facts:
\begin{itemize}
\item
Let $p \in D_{\pi\vert_X}$ be a point s.t.\ $\rho$ is not an isomorphism at $p$. Let $\rho_0 : V' \to V$ be the blow up at $p$ and $(\pi': W' \to V', X')$ the induced ramified covering. Then $\rho$ factors as $\rho = \rho_1 \rho_0$ where $\rho_1 : V_{\rho} \to V'$ is an embedded desingularization of $D_{\pi'\vert_{X'}}$, $X_{\rho} \to X$ factors through $X' \to X$ and also $\pi'^{-1}(D_{\pi'\vert_{X'}})$ must be equal to the support of the pullback of $\SS$. So it is equivalent to compute the formal description of $\Pi_{\rho_1}$ above $D_{\pi'\vert_{X'}}$.% and then compose all the formal prime divisors with the morphism given by the first blow up $\rho_0$. 
\item
Now $\rho_0$ is an isomorphism everywhere except at $p$. We can split the computation of the formal desingularization $\Pi_{\rho}$ into two parts; the computation of formal prime divisors which are not centered above $p$ on the one hand and those which are centered above $p$, or equivalently, whose pullbacks (see Corollary~\ref{cor.DivPullback}) are centered above the exceptional divisor $\rho_0^{-1}(p)$. Computing the latter will be delegated to a recursive call.
\item
When blowing up the (not necessarily rational) point $p$ we may first apply a morphism to the projection plane that induces a formal isomorphism at $p$ because of Lemma~\ref{lem.NormCompl} and Lemma~\ref{lem.BlowCompl}, compare Remark~\ref{rem.FormalIsos} below.
\end{itemize}
\end{rem}

Finally, Example~\ref{run.FormalDesing} has shown that in the case of nicely ramified coverings we have to compute formal prime divisors in two ways: Those which are centered above the components of the discriminant curve are obtained using Lemma~\ref{lem.RatParDecomp} with $n=1$, and those which are centered above the normal crossings of the discriminant curve using a combination of Lemma~\ref{lem.RatParDecomp} with $n=2$, Lemma~\ref{lem.ToricSing} and Lemma~\ref{lem.BlowCompl}.

\subsection{The Algorithm}

In the following algorithmic descriptions we allow subsets of a set $A$, which in our notation will be elements of $2^A$, as data types. These will either be finite sets or they will be finitely generated ideals of a ring $A$. So it is clear that they can be represented. For simplicity of notation, we also allow passing of homomorphisms from polynomial rings in a finite number of variables. They can obviously be represented by the images of their generators. We assume that we can represent power series which will be explained later in Section~\ref{sec.AlgSeries}. If $\phi: A \to B$ is a homomorphism of rings, we write again $\phi\lift{w}: A[w] \to B[w]$ for its lifting to the corresponding polynomial rings by coefficient-wise application and $\map_\phi: 2^A \to 2^B$ for the function on subsets defined by element-wise application.

Let $X \subset \P_{\E}^3$ be a closed hypersurface. Following Section~\ref{sec.RamCover}, we first have to produce a ramified covering. This is done in algorithm {\tt DesingGlobal}. By $\overline{\E}$ (resp.\ $\overline{\E(s)}$) we denote the algebraic closure of $\E$ (resp.\ of the rational function field). If it shows up in the return type of a signature we actually mean that the result involves some finite field extension (of transcendence degree $1$), {\it i.e.}, we do not rely on a system for computing with algebraic closures.

The algorithm will be formulated for reduced surfaces, see Remark~\ref{rem.DesingReduced}. This is due to the fact that we will transform our surface by extending the base field in certain steps and cannot assure that the transformed surface remains integral, even when the original surface was, see Remark~\ref{rem.FormalIsos} below.

\begin{algorithm}[H]
\caption{$\texttt{DesingGlobal}(F : \E[x_0,\dots,x_3]) : 2^{\E[x_0,\dots,x_3] \to \overline{\E(s)}\ll t\rr}$}\label{alg.DesingGlobal}
\begin{algorithmic}[1]
\REQUIRE A squarefree homogeneous polynomial $F \neq 0$.
\ENSURE  A finite set of homomorphisms $\E[x_0,\dots,x_3] \to \F\ll t\rr$ factoring through $\E[x_0,\dots,x_3]/\langle F \rangle$ s.t.\ the induced morphisms $\spec \F\ll t\rr  \to \proj \E[x_0,\dots,x_3]/\langle F \rangle$ are a formal desingularization.
\STATE Let $\phi: \E[x_0,\dots,x_3] \to \E[x_0,\dots,x_3]$ be a linear automorphism s.t.\ $\phi(F)(1,0,0,0) \neq 0$;
\STATE $\psi_1: \E[x_0,\dots,x_3] \to \E[u,v][w]: x_0 \mapsto w, x_1 \mapsto 1, x_2 \mapsto u, x_3 \mapsto v$;
\STATE $S := \map_{\rho \mapsto \rho \psi_1 \phi}\texttt{DesingLocal}(\psi_1 \phi(F), \langle 0 \rangle)$;
\STATE $\psi_2: \E[x_0,\dots,x_3] \to \E[u,v][w]: x_0 \mapsto w, x_1 \mapsto v, x_2 \mapsto 1, x_3 \mapsto u$;
\STATE $S := S \cup \map_{\rho \mapsto \rho \psi_2 \phi}\texttt{DesingLocal}(\psi_2 \phi(F), \langle v \rangle)$;
\STATE $\psi_3: \E[x_0,\dots,x_3] \to \E[u,v][w]: x_0 \mapsto w, x_1 \mapsto u, x_2 \mapsto v, x_3 \mapsto 1$;
\STATE $S := S \cup \map_{\rho \mapsto \rho \psi_3 \phi}\texttt{DesingLocal}(\psi_3 \phi(F), \langle u, v \rangle)$;
\STATE \return\ $S$;
\end{algorithmic}
\end{algorithm}

In line $1$ we choose a linear automorphism of $\P_{\E}^3$ (represented by $\phi$) s.t.\ the preimage of $X$ under this automorphism is Noether normalized by a projection onto the plane $x_0 = 0$.

\begin{rem}
The automorphism maps $(1,0,0,0)$ to a point $p \not\in X$. In an actual implementation one should find $p$ s.t.\ most of its coordinates are zero and the rest are small integers. This preserves sparsity in $\phi(F)$ and keeps coefficients small.
\end{rem}

For convenience of description, we will actually assume that $\phi$ is the identity. Then the ramified covering $(\pi: W \to V, X)$ is given as follows: We set $W := \P_{\E}^3 \setminus \{ (1:0:0:0) \}$, $V$ is the plane $x_0 = 0$, $\pi$ is the corresponding linear projection and $X$ is defined by the vanishing of $F$ monic in $x_0$.

The algorithm produces a set of homomorphisms $S$ representing a formal desingularization. Therefore we cover $V$ by open subsets $x_i \neq 0$ (given by the $\psi_i$) for $1 \le i \le 3$ and call algorithm {\tt DesingLocal} for each of those in lines $2$ to $7$. The latter algorithm produces formal desingularizations of the respective affine subsets. Because of the huge overlaps we add \emph{focus} ideals to each call.

\begin{algorithm}[H]
\caption{$\texttt{DesingLocal}(f : \E[u,v][w], \FF: 2^{\E[u,v]}) : 2^{\E[u,v][w] \to \overline{\E(s)}\ll t\rr}$}\label{alg.DesingLocal}
\begin{algorithmic}[1]
\REQUIRE A squarefree polynomial $f \neq 0$, monic in $w$, and a focus ideal $\FF$.
\ENSURE  A finite set of homomorphisms $\E[u,v][w] \to \F\ll t\rr$ factoring through $\E[u,v][w]/\langle f \rangle$ s.t.\ the induced morphisms $\spec \F\ll t\rr  \to \spec \E[u,v][w]/\langle f \rangle$ form the subset of a formal desingularization which is centered above the closed subset defined by $\FF$.
\STATE $d := \texttt{SquareFreePart}(\disc_w(f))$; $E := \texttt{IrredFactors}(d)$;
%\STATE $S := \bigcup\{\texttt{DivisorsAboveCurve}(f, e) \mid e \in E \text{ with } \FF \subseteq \langle e \rangle\}$;
\STATE $S := \bigcup_{e \in E \text{ with } \FF \subseteq \langle e \rangle} \texttt{DivisorsAboveCurve}(f, e)$;
\FOR{$(u_0,v_0) \in \texttt{ZeroSet}(\FF + \langle d, \partial d / \partial u, \partial d / \partial v \rangle)$}
\STATE $\psi: \E[u,v] \to \E'[u,v]: u \mapsto u + u_0, v \mapsto v + v_0$;
%\STATE $S := S \cup \map_{\rho \mapsto \rho \psi\lift{w}}\texttt{DesingRec}(\psi\lift{w}(f), \map_\psi(\{e \in E \mid e(u_0,v_0) = 0\}))$;
\STATE $S := S \cup \map_{\rho \mapsto \rho \psi\lift{w}}\texttt{DesingRec}(\psi\lift{w}(f), \{\psi(e) \mid e \in E \text{ and } \psi(e)(0,0) = 0\})$;
\ENDFOR
\STATE \return\ $S$;
\end{algorithmic}
\end{algorithm}

In line $1$ we compute the defining equation $d$ of the \emph{reduced} discriminant curve and its factors $E$. For all $e \in E$ the prime ideal $\langle e \rangle$ corresponds to the generic point $p \in V$ of an irreducible component of the discriminant curve.

No matter how the desingularization $\rho: V_{\rho} \to V$ of the discriminant curve looks like, it is always a succession of point blow ups. Therefore $p$ will be isomorphically contained in $V_\rho$. The same holds for points in the normalization of $X$ lying above $p$. According to our paradigm in Remark~\ref{rem.DivAndConqPara} we compute the formal prime divisors centered above these $p$ already at this stage of the algorithm, see line $2$, by calling algorithm {\tt DivisorsAboveCurve} (see Algorithm~\ref{alg.DivisorsAboveCurve}) for each $e \in E$ which is in focus.

In line $3$ we compute the singular locus of the reduced discriminant curve, more precisely, that part which is in focus. For each of its closed points $p$ we want to find the formal prime divisors centered above $p$. Therefore we apply the homomorphism $\psi$ of line $4$ (which corresponds to moving $p$ to the origin) and then call algorithm {\tt DesingRecursive} in line $5$.

\begin{rem}[Exploiting Formal Isomorphisms]\label{rem.FormalIsos}
At this point we would like to mention that $\psi$ may involve an algebraic field extension, namely, by the residue field $\E'$ of $p$. Therefore $\psi$, which looks like a mere translation, is not an isomorphism. Nevertheless the induced morphism of schemes $\spec \E'[u,v] \to \spec \E[u,v]$ is formally isomorphic at $p$ by Corollary~\ref{cor.PtCompl}. (Note that we can choose a system $f_1,f_2 \in \E[u,v]$ of generators of the maximal ideal corresponding to $p$ s.t.\ $f_1 \in \E[u]$ is irreducible and the image of $f_2$ in $(\E[u]/\langle f_1 \rangle)[v]$ is also irreducible.) But for computing a formal desingularization, we may well pass to a formally isomorphic scheme as a consequence of Lemma~\ref{lem.NormCompl} and Lemma~\ref{lem.BlowCompl}. Note that the introduction of a field extension may further split the defining equation, {\it i.e.}, $\psi\lift{w}(f)$ may be reducible even though $f$ is not. Also the discriminant factors might split again. This is another (more important) reason for computing the formal prime divisors above the components of the discriminant already in line $2$.
\end{rem}

\begin{rem}[Auxiliary Functions]\label{rem.AuxFunc}
The above algorithm depends on a couple of auxiliary functions which we are not giving in detail, their names are mainly self-explanatory: {\tt SquareFreePart} should compute the squarefree part of a polynomial and {\tt IrredFactors} is supposed to produce the set of irreducible factors of a polynomial. A comment on {\tt ZeroSet} is in order. It expects a zero-dimensional ideal $\FF \subseteq \E[x_1,\dots,x_n]$. It should return a finite set of $n$-tuples s.t.\ for each maximal ideal $\mathfrak{f}$ containing $\FF$ there is exactly one tuple $(\xi_1, \dots, \xi_n) \in (\E')^n$ s.t.\ $\E \subseteq \E'$ and the induced homomorphism $\E[x_1,\dots,x_n] \to \E': x_i \mapsto \xi_i$ lifts to an isomorphism from $\E[x_1,\dots,x_n]/\mathfrak{f}$.
\end{rem}

\begin{rem}
Computing the squarefree part and factorization was done in Algorithm~\ref{alg.DesingLocal} to keep the number of parameters small. Of course it would fit more naturally in Algorithm~\ref{alg.DesingGlobal} to avoid multiple computations.
\end{rem}

The recursive algorithm {\tt DesingRecursive} now implements the \emph{divide and conquer} paradigm, see again Remark~\ref{rem.DivAndConqPara}.

\begin{algorithm}[H]
\caption{$\texttt{DesingRecursive}(f : \E[u,v][w], E : 2^{\E[u,v]}) : 2^{\E[u,v][w] \to \overline{\E(s)}\ll t\rr}$}\label{alg.DesingRecursive}
\begin{algorithmic}[1]
\REQUIRE A squarefree polynomial $f \neq 0$, monic in $w$, and a set of polynomials $E$ s.t.\ $e(0,0) = 0$ for all $e \in E$ and $\prod_{e \in E}e$ is the squarefree part of $\disc_w(f)$ in the local ring $\E[u,v]_{\langle u,v \rangle}$.
\ENSURE  A finite set of homomorphisms $\E[u,v][w] \to \F\ll t\rr$ factoring through $\E[u,v][w]/\langle f \rangle$ s.t.\ the induced morphisms $\spec \F\ll t\rr  \to \spec \E[u,v][w]/\langle f \rangle$ form the subset of a formal desingularization which is centered above the origin.
\IF{$\texttt{IsNormalCrossing}(E)$}
\STATE \return\ $\texttt{DivisorsAboveCrossing}(f,E)$;
\ENDIF
\STATE $\phi_u: \E[u,v] \to \E[u,v]: u \mapsto uv, v \mapsto v$; $\phi_v: \E[u,v] \to \E[u,v]: u \mapsto v, v \mapsto uv$;
\STATE $f_u := \phi_u\lift{w}(f)$; $f_v := \phi_v\lift{w}(f)$;
\STATE $E_u := \map_{e \mapsto \phi_u(e)/v^{\ord(e)}}(E)$; $E_v := \map_{e \mapsto \phi_v(e)/v^{\ord(e)}}(E)$;
\STATE $S := \map_{\rho \mapsto \rho \phi_u\lift{w}}\texttt{DivisorsAboveCurve}(f_u, v)$;
\FOR{$(u_0,v_0) \in \texttt{ZeroSet}(\langle \prod_{e \in E_u}e, v \rangle)$}
\STATE $\psi: \E[u,v] \to \E'[u,v]: u \mapsto u + u_0, v \mapsto v + v_0$;
\STATE $S := S \cup \map_{\rho \mapsto \rho \psi \lift{w} \phi_u \lift{w}} \texttt{DesingRecursive}(\psi\lift{w}(f_u),$ \\
\mbox{}\hspace{5cm}$\{\psi(e) \mid e \in E_u \text{ and } \psi(e)(0,0) = 0\} \cup \{v\})$;
\ENDFOR
\IF{exists $e \in E_v$ s.t.\ $e(0,0) = 0$}
\STATE $S := S \cup \map_{\rho \mapsto \rho \phi_v \lift{w}}\texttt{DesingRecursive}(f_v,$ \\
\mbox{}\hspace{5cm}$\{e \mid e \in E_v \text{ and } e(0,0) = 0\} \cup \{v\})$;
\ENDIF
\STATE \return\ $S$;
\end{algorithmic}
\end{algorithm}

In line $1$ we check, whether the origin of the reduced discriminant curve is a normal crossing. If this is the case we know (see Corollary~\ref{cor.ToroidalSing}) that the points above the origin in the normalization are toroidal singularities that can be desingularized by a succession of blow ups in isolated singular points. The corresponding formal prime divisors are computed by algorithm {\tt DivisorsAboveCrossing} (see Algorithm~\ref{alg.DivisorsAboveCrossing}) and returned.

Otherwise we have to modify the discriminant curve by blowing up the origin. The two charts of the blow up are computed in lines $3$ to $5$: We determine the defining equations $f_u$ and $f_v$ of the transformed surface and also modify the discriminant factors accordingly.

Note that the homomorphisms are such that the exceptional divisor is given by $v=0$ in \emph{both} charts. Over this exceptional divisor there are centered a couple of formal prime divisors. By the same reasoning as for Algorithm~\ref{alg.DesingLocal} (see Remark~\ref{rem.FormalIsos}) we compute them immediately by calling {\tt DivisorsAboveCurve} in line $6$.

Now we have to consider the points on the exceptional divisor which are singular. In one of the charts they are given by the set of line $7$. As in Algorithm~\ref{alg.DesingLocal} we move these points to the origin and go into recursion, see lines $8$ and $9$. Now there is possibly one singular point left to consider, namely, the origin of the other chart. It is treated in line $11$.

To complete the algorithm, it remains to show how to compute the formal prime divisors which are centered above irreducible components or normal crossings of the reduced discriminant curve.

\subsubsection{Divisors above Generic Points of the Discriminant Curve}

We now give an algorithm for computing the formal prime divisors which are centered above generic points of the discriminant curve. This is easy by what we have developed so far.

\begin{algorithm}[H]
\caption{$\texttt{DivisorsAboveCurve}(f : \E[u,v][w], e : \E[u,v]) : 2^{\E[u,v][w] \to \overline{\E(s)}\ll t\rr}$}\label{alg.DivisorsAboveCurve}
\begin{algorithmic}[1]
\REQUIRE A squarefree polynomial $f \neq 0$, monic in $w$, and an irreducible factor $e$ of $\disc_w(f)$.
\ENSURE  A finite set of homomorphisms $\E[u,v][w] \to \F\ll t\rr$ factoring through $\E[u,v][w]/\langle f \rangle$ s.t.\ the induced morphisms $\spec \F\ll t\rr  \to \spec \E[u,v][w]/\langle f \rangle$ form the subset of a formal desingularization which is centered above $\langle e \rangle$.
\STATE $\F_0 := \totfrac(\E[u,v]/e)$; Let $u_0, v_0 \in \F_0$ be the natural images of $u, v$;
\IF{$\partial e / \partial v \neq 0$}
\STATE $\phi: \E[u,v] \to \F_0[t]: u \to u_0, v \to v_0 + t$;
\ELSE
\STATE $\phi: \E[u,v] \to \F_0[t]: u \to u_0 + t, v \to v_0$;
\ENDIF
\STATE $S := \emptyset$;
\FOR{$(\sigma, \alpha) \in \texttt{Param}(\phi\lift{w}(f))$}
\STATE %$\sigma: \F_0\ll t\rr \to \F\ll t\rr: t \mapsto \sigma(t)$;
$\psi: \F\ll t^{(1/e)\Z}\rr \to \F\ll t\rr: \gamma \mapsto \texttt{Evaluate}(\gamma, (e), t)$;
\COMMENT{assuming $\alpha \in \F\ll t^{(1/e)\Z}\rr$}
\STATE $S := S \cup \{\E[u,v][w] \to \F\ll t\rr : u \mapsto \psi\sigma\phi(u), v \mapsto \psi\sigma\phi(v), w \mapsto \psi(\alpha) \}$;
\ENDFOR
\STATE \return\ $S$;
\end{algorithmic}
\end{algorithm}

In lines $1$ to $5$ we construct a homomorphism $\phi: \E[u,v] \to \F_0[t]$ inducing an isomorphism from the completed localization at $\langle e \rangle$ to $\F_0\ll t \rr$, see Corollary~\ref{cor.PtCompl}. Therefore the completions of the localizations of $\intclos{\E[u,v][w]/\langle f \rangle}$ at prime ideals above $\langle e \rangle$ are isomorphic to those of $\intclos{\F_0\ll t\rr/\langle \phi\lift{w}(f) \rangle}$ above $\langle t \rangle$ by Lemma~\ref{lem.NormCompl}. Lemma~\ref{lem.RatParDecomp} tells us how these completions can be computed using a complete set of rational parametrizations. This is done in lines $7$ to $9$ using the results of a call to {\tt Param} (see Algorithm~\ref{alg.Param} in Section~\ref{sec.RatParams}). The homomorphism $\psi$ defined via {\tt Evaluate} (see Algorithm~\ref{alg.Evaluate}) is just mapping $t^{1/e} \mapsto t$ for cosmetic reasons, {\it i.e.}, getting rid of denominators.

\subsubsection{Divisors above Normal Crossings of the Discriminant Curve}

First we fill a gap in Algorithm~\ref{alg.DesingRecursive} and show how to test the normal crossing property for a set of curves.

\begin{algorithm}[H]
\caption{$\texttt{IsNormalCrossing}(E : 2^{\E[u,v]}) : \texttt{Boolean}$}\label{alg.IsNormalCrossing}
\begin{algorithmic}[1]
\REQUIRE A set of squarefree polynomials $E$ s.t.\ $e(0,0) = 0$ for all $e \in E$.
\ENSURE  \texttt{true} iff the curves defined by $E$ are \emph{considered} normal crossing at the origin.
\IF{$E$ is not of the form $\{v, e\}$}
\STATE \return\ \texttt{false};
\ELSE
\STATE \return\ $\partial e/\partial u (0,0) \neq 0$;
\ENDIF
\end{algorithmic}
\end{algorithm}

Note that in general $E$ describes a set of curves with normal crossing at the origin when $E = \{e_1, e_2\}$ and $\det(\partial (e_1, e_2)/\partial (u,v)) (0,0) \neq 0$. But this algorithm returns {\tt true} only for the special situation that one of the curves is actually $v=0$.

\begin{rem}[Almost Minimal Jung Desingularizations]\label{rem.AlmostMinimal}
Although Algorithm~\ref{alg.IsNormalCrossing} does in fact test for normal crossings of a special form the overall algorithm will eventually terminate; Computing the point blow ups as in Algorithm~\ref{alg.DesingRecursive} guarantees that the exceptional divisor is always of the form $v=0$ in both charts. It is known (see \cite[Thm.\ 1.47]{JK:2007}) that one can compute an embedded desingularization by a finite number of point blow ups. If one of the normal crossings is not of the above form, then our algorithm computes an additional blow up but terminates at the next level. In other words, the computed Jung desingularization belongs to an \emph{almost minimal} embedded desingularization of the discriminant curve. The benefit is that the homomorphism in line $2$ of the next algorithm is so easy to compute. One could also use usual normal crossings and compute more complicated homomorphisms.
\end{rem}

\begin{algorithm}[H]
\caption{$\texttt{DivisorsAboveCrossing}(f : \E[u,v][w], E : 2^{\E[u,v]}) : 2^{\E[u,v][w] \to \overline{\E}(s)\ll t\rr}$}\label{alg.DivisorsAboveCrossing}
\begin{algorithmic}[1]
\REQUIRE A squarefree polynomial $f \neq 0$, monic in $w$, and a set of polynomials $E = \{v, e\}$ s.t.\ $e(0,0) = 0$, $\partial e/\partial u (0,0) \neq 0$ and $ve$ is the squarefree part of $\disc_w(f)$ in the local ring $\E[u,v]_{\langle u,v \rangle}$.
\ENSURE  A finite set of homomorphisms $\E[u,v][w] \to \F\ll t\rr$ factoring through $\E[u,v][w]/\langle f \rangle$ s.t.\ the induced morphisms $\spec \F\ll t\rr  \to \spec \E[u,v][w]/\langle f \rangle$ form the subset of a formal desingularization which is centered above the origin.
\STATE Let $g := e(z,v') - u' \in \E[u',v'][z]$;
\STATE $\phi: \E[u,v] \to \E\ll u', v'\rr: v \mapsto v', u \mapsto \texttt{ImplicitFunction}(g)$;
\STATE $S := \emptyset$;
\FOR{$(\sigma, \alpha) \in \texttt{Param}(\phi\lift{w}(f))$}
\STATE Let $\mfn_1, \dots, \mfn_l \in \Z^2$ be the sequence of generators of $\Gamma^{\vee} \cap \R_{\ge 0}^2$;\\
\COMMENT{assuming $\alpha \in \E'\ll (u',v')^{\Gamma}\rr$ and ordering generators as in Figure~\ref{fig.Lattices}}
\FOR{$1 \le i \le l-2$}
\STATE $\psi: \E'\ll(u',v')^{\Gamma}\rr \to \E'(s)\ll t \rr: \gamma \mapsto \texttt{Evaluate}(\gamma, (\mfn_i, \mfn_{i+1}), (s, t))$;\\
\COMMENT{where $s,t \in \E'(s)\ll t \rr$}
\STATE $S := S \cup \{ \E[u,v][w] \to \F\ll t\rr : u \mapsto \psi \sigma \phi(u), v \mapsto \psi \sigma \phi(v), w \mapsto \psi(\alpha) \}$;
\ENDFOR
\ENDFOR
\STATE \return\ $S$;
\end{algorithmic}
\end{algorithm}

This algorithm expects that the discriminant curve of the input surface has a normal crossing at the origin, more precisely, the discriminant of the defining equation is $v^{d_1}e^{d_2}$ times a local unit. We want to compute bivariate parametrizations above the origin, but $f$ is not yet normal crossing. We first have to apply a formal isomorphism $\phi$, see line $2$, that maps $v \mapsto v'$ and $e \mapsto u'$, because then the discriminant becomes $v'^{d_1}u'^{d_2}$ up to a unit and the defining equation $\phi\lift{w}(f)$ is quasi-ordinary. To this end the image of $u$ must fulfill the equation $e(\phi(u),v') - u' = 0$ which has a unique solution by the implicit function theorem and is computed by a call to {\tt ImplicitFunction} (see Algorithm~\ref{alg.ImplicitFunction} in Section~\ref{sec.Expand}).

By the same reasoning as for Algorithm~\ref{alg.DivisorsAboveCrossing} the completions of the localizations of $\intclos{\E[u,v][w]/\langle f \rangle}$ at prime ideals above the origin can be computed by bivariate rational parametrization using a call to {\tt Param}, see line $4$. Now assume such a completion is given by $\E'\ll (u',v')^{\Gamma} \rr$.

We know how to compute a formal description of a special desingularization of $\spec \E'[(u',v')^{\Gamma}]$, namely, the one obtained by a minimal number of point blow ups, see Lemma~\ref{lem.ToricSing} and Remark~\ref{rem.MinimalDesing}. We compute the corresponding homomorphisms $\E'\ll(u',v')^{\Gamma}\rr \to \E'(s)\ll t\rr: (u',v')^{\mfm} \mapsto s^{\mfn_i(\mfm)} t^{\mfn_{i+1}(\mfm)}$ by calls to {\tt Evaluate} in line $7$. This desingularization commutes with completion by Lemma~\ref{lem.BlowCompl}, hence, we get a formal desingularization of the toroidal singularity by composing these homomorphisms with the one given by the parametrization, compare line $8$.

%%% Local Variables: 
%%% TeX-master: "~/doc/work/academic/worked_texts/Formal-Desing/formal-desings.tex"
%%% End: 

%% file: lattice.030.tex
\begin{picture}(0,0)%
\includegraphics{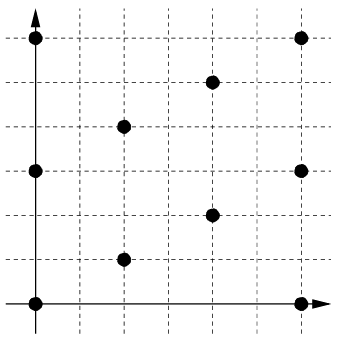}%
\end{picture}%
\setlength{\unitlength}{1243sp}%
\begingroup\makeatletter\ifx\SetFigFont\undefined%
\gdef\SetFigFont#1#2#3#4#5{%
  \reset@font\fontsize{#1}{#2pt}%
  \fontfamily{#3}\fontseries{#4}\fontshape{#5}%
  \selectfont}%
\fi\endgroup%
\begin{picture}(5510,5203)(2161,-7436)
\put(2161,-2671){\makebox(0,0)[lb]{\smash{{\SetFigFont{11}{13.2}{\rmdefault}{\mddefault}{\updefault}{\color[rgb]{0,0,0}$1$}%
}}}}
\put(6841,-7306){\makebox(0,0)[lb]{\smash{{\SetFigFont{11}{13.2}{\rmdefault}{\mddefault}{\updefault}{\color[rgb]{0,0,0}$1$}%
}}}}
\put(2161,-7306){\makebox(0,0)[lb]{\smash{{\SetFigFont{11}{13.2}{\rmdefault}{\mddefault}{\updefault}{\color[rgb]{0,0,0}$0$}%
}}}}
\put(6121,-2581){\makebox(0,0)[lb]{\smash{{\SetFigFont{11}{13.2}{\rmdefault}{\mddefault}{\updefault}{\color[rgb]{0,0,0}$\Gamma$}%
}}}}
\end{picture}%

%% file: lattice-dual.030.tex
\begin{picture}(0,0)%
\includegraphics{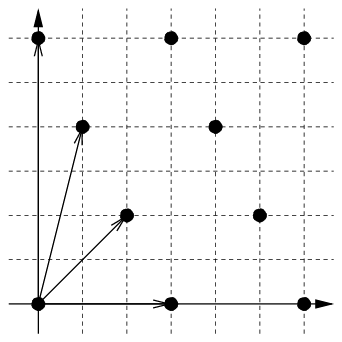}%
\end{picture}%
\setlength{\unitlength}{1243sp}%
\begingroup\makeatletter\ifx\SetFigFont\undefined%
\gdef\SetFigFont#1#2#3#4#5{%
  \reset@font\fontsize{#1}{#2pt}%
  \fontfamily{#3}\fontseries{#4}\fontshape{#5}%
  \selectfont}%
\fi\endgroup%
\begin{picture}(5555,5251)(2116,-7436)
\put(6841,-7306){\makebox(0,0)[lb]{\smash{{\SetFigFont{11}{13.2}{\rmdefault}{\mddefault}{\updefault}{\color[rgb]{0,0,0}$6$}%
}}}}
\put(2161,-2671){\makebox(0,0)[lb]{\smash{{\SetFigFont{11}{13.2}{\rmdefault}{\mddefault}{\updefault}{\color[rgb]{0,0,0}$6$}%
}}}}
\put(2161,-7306){\makebox(0,0)[lb]{\smash{{\SetFigFont{11}{13.2}{\rmdefault}{\mddefault}{\updefault}{\color[rgb]{0,0,0}$0$}%
}}}}
\put(6121,-2581){\makebox(0,0)[lb]{\smash{{\SetFigFont{11}{13.2}{\rmdefault}{\mddefault}{\updefault}{\color[rgb]{0,0,0}$\Gamma^{\vee}$}%
}}}}
\put(2701,-4471){\makebox(0,0)[lb]{\smash{{\SetFigFont{11}{13.2}{\rmdefault}{\mddefault}{\updefault}{\color[rgb]{0,0,0}$\mfn_2$}%
}}}}
\put(2116,-3301){\makebox(0,0)[lb]{\smash{{\SetFigFont{11}{13.2}{\rmdefault}{\mddefault}{\updefault}{\color[rgb]{0,0,0}$\mfn_1$}%
}}}}
\put(3421,-5416){\makebox(0,0)[lb]{\smash{{\SetFigFont{11}{13.2}{\rmdefault}{\mddefault}{\updefault}{\color[rgb]{0,0,0}$\mfn_3$}%
}}}}
\put(4096,-6676){\makebox(0,0)[lb]{\smash{{\SetFigFont{11}{13.2}{\rmdefault}{\mddefault}{\updefault}{\color[rgb]{0,0,0}$\mfn_4$}%
}}}}
\end{picture}%

%% file: series.tex
In order to implement the above algorithms we have to represent multivariate, fractionary, algebraic power series, {\it i.e.}, power series with coefficients in a field $\E$, variables $x_1,\dots,x_n$ and exponents in $\Gamma \cap \R_{\ge 0}^n$ (the non-negative orthant of a full rational lattice) that are roots of non-zero polynomials in $\E[x_1,\dots,x_n][z]$. In the sequel we denote this ring by $\E\ll\xu^{\Gamma}\rr$ (hence, this notation from now means \emph{algebraic}, not formal power series). We must be able to perform a couple of operations:

\begin{enumerate}
\item We have to expand power series up to arbitrary order.
\item We have to compute the power series arising in complete sets of rational parametrizations for quasi-ordinary polynomials and solve equations fulfilling the conditions of the implicit function theorem.
\item Let $\hom(\Gamma,\E^*)$ denote the (commutative, multiplicative) group of homomorphisms from the lattice $\Gamma$ to the multiplicative group $\E^*$ with group operation given by multiplication in the codomain. With $\sigma \in \hom(\Gamma,\E^*)$ we associate the automorphism $\E\ll\xu^{\Gamma}\rr  \to \E\ll\xu^{\Gamma}\rr$ mapping $\xu^{\mfm} \mapsto \sigma(\mfm)\xu^{\mfm}$ and $c \mapsto c$ for $c \in \E$. This action of $\hom(\Gamma,\E^*)$ on $\E\ll\xu^{\Gamma}\rr$ should be computable.
\item If $\alpha \in \E_1\ll\xu^{\Gamma_1}\rr$ is a power series, $\mfn_1, \dots, \mfn_l$ are vectors in $\Gamma_1^{\vee} \cap \R_{\ge 0}^m$ and $(\xi_1, \dots, \xi_l)$ are power series in $\E_2\ll\xu^{\Gamma_2}\rr$ where $\E_2$ is a field extension of $\E_1$, we want to compute the image of $\alpha$ under the homomorphism $\phi: \E_1\ll\xu^{\Gamma_1}\rr \to \E_2\ll\xu^{\Gamma_2}\rr$ which maps $\xu^{\mfm} \mapsto \prod_{1 \le i \le l} \xi_i^{\mfn_i(\mfm)}$ and $c \mapsto c$ for $c \in \E_1$ if such $\phi$ is well-defined.
\end{enumerate}

\begin{rem}[Origin of Requirements]
Requirement (1) just means that (in order to return values) we need a representation in finite terms of series which a priori are infinite objects.

Requirement (2) is obvious. We have to compute rational parametrizations in Algorithms~\ref{alg.DivisorsAboveCurve} and \ref{alg.DivisorsAboveCrossing} in order to compute formal prime divisors in the case of nicely ramified coverings.  The first line of the latter algorithm also implies solving an equation that fulfills the conditions of the implicit function theorem.

Further it would be nice to express certain homomorphisms $\E_1\ll\xu^{\Gamma_1}\rr \to \E_2\ll\xu^{\Gamma_2}\rr$. In general these are given by specifying the images of generators $\xu^{\mfm}$ in a consistent way. But such general homomorphisms are not easy to compute. Instead we concentrate on the two special cases (3) and (4).

Computing the action induced by lattice homomorphisms becomes relevant when computing rational parametrizations. So we need (3) to accomplish (2).

In line $7$ of Algorithm~\ref{alg.DivisorsAboveCrossing}, we find a transformation which is a special case
of requirement (4). Another special case of that requirement is substitution of power series $\xi_i \in \E_2\ll\xu^{\Gamma_2}\rr$ into a polynomial or power series $g \in \E_1\ll x_1,\dots,x_n\rr$ which is an instance with $\Gamma_1 = \Gamma_1^{\vee} = \Z^n$ and $\mfn_i$ the standard basis vectors. Such substitutions occur when we compute the composite of two homomorphisms, {\it e.g.}, in line $9$ of Algorithm~\ref{alg.DivisorsAboveCurve}.

Also note that an even more special case of the last requirement is effectivity of addition and multiplication (with $g = x_1 + x_2$ resp.\ $g = x_1 x_2$).
\end{rem}

Computations with algebraic power series usually involve studying the support and the \emph{Newton Polygon} of the defining equation (see \cite{RW:1978,JM:1995,FB_FRJ:2001,JvdH:2006}), which live in $\Q^{n} \times \Z$ resp.\ $\R^{n+1}$. In order to allow for a nice implementation, we will consider a \emph{flattened} support by assuming that $\Q^n$ is ordered as a group and considering $\Q^{n} \times \Z$ a ``plane''. In this setting, one can do a fair amount of ``convex geometry''.

In fact we will work with two different orderings. Let $q_1,q_2 \in \Q^n$. We have a partial ordering on $\Q^n$ by comparing the total degrees $\lvert q_1 \rvert$ and $\lvert q_2 \rvert$ ({\it i.e.}, the sum of their components) via $\le$ as rational numbers. Further we assume a total ordering $\cle$ which is a refinement of it, {\it i.e.}, $\lvert q_1 \rvert \le \lvert q_2 \rvert$ implies $q_1 \cle q_2$. For a power series $0 \neq \alpha = \sum_{\mfm} \alpha_{\mfm} \xu^{\mfm} \in \E\ll\xu^{\Gamma}\rr$ we can therefore define two supports and two orders, namely, $\supp_{\le}(\alpha) := \{\lvert \mfm \rvert \in \Q \mid \alpha_{\mfm} \neq 0\}$, $\supp_{\cle}(\alpha) := \{\mfm \in \Q^n \mid \alpha_{\mfm} \neq 0\}$, $\ord_{\le}(\alpha) := \min_{\le}(\supp_{\le}(\alpha))$ and $\ord_{\cle}(\alpha) := \min_{\cle}(\supp_{\cle}(\alpha))$. W.r.t.\ the finer ordering, we will also need the initial term $\lt_{\cle}(\alpha) := \alpha_{\ord_{\cle}(\alpha)} \xu^{\ord_{\cle}(\alpha)}$. The degree compatibility is important because it implies that we can easily expand series up to an arbitrary order w.r.t\ the fine ordering if and only if we can expand up to arbitrary total degree.

We also have to study polynomials with power series coefficients. Consider \[0 \neq g = \sum_{0 \le i \le d} \left(\sum_{\mfm} g_{\mfm,i} \xu^{\mfm}\right) z^i \in \E_0\ll\xu^{\Gamma_0}\rr[z].\] Its \emph{support} is $\supp_{\cle}(g) := \{(\mfm,i) \in \Q^n \times \Z \mid g_{\mfm,i} \neq 0\}$. 
We define linear maps \[\varphi_{\mfn}: \Q^n \times \Z \to \Q^n: (\mfm,i) \mapsto \mfm + i \mfn\] for all $\mfn \in \Q_{\ge 0}^n$ and set $\ord_{\cle,\mfn}(g) := \min_{\cle} \{\varphi_{\mfn}(\mfm,i) \mid (\mfm,i) \in \supp_{\cle}(g)\}$. We say that $\supp_{\cle}(g)$ has an \emph{non-trivial} edge of slope $\mfn$ iff $\varphi_{\mfn}$ attains its minimum for \emph{at least two} arguments in $\supp_{\cle}(g)$. Further we define the $\mfn$-th \emph{edge equation} as \[\edge_{\cle,\mfn}(g) := \mathop{\sum_{(\mfm,i) \in \supp_{\cle}(g) \text{ and}}}_{\varphi_{\mfn}(\mfm,i) = \ord_{\cle,\mfn}(g)} g_{\mfm,i} \xu^{\mfm} z^i.\] Then $\supp_{\cle}(g)$ has a non-trivial edge of slope $\mfn$ iff the $\mfn$-th edge equation is not a single term, {\it i.e.}, it is also non-trivial.

Note that $\supp_{\cle}(g)$ has at most finitely many non-trivial edges. Indeed, assume $g = \sum_{l \le i \le k} \gamma_i z^i$ with $\gamma_{l} \neq 0$ and $\gamma_{k} \neq 0$ and let $\mfo_i := \ord_{\cle}(\gamma_i)$ whenever $\gamma_i \ne 0$. Then the only possible slopes are $(i-j)^{-1}(\mfo_j-\mfo_i)$ for $i > j$ and $\mfo_j \cgt \mfo_i$ with $\gamma_i \ne 0$, $\gamma_j \ne 0$.
Denoting the slopes that occur by $\mfn_m$ we could define the ``Newton Polygon'' of $g$ as the set $\{(\mfm,i) \in \Q^n \times \Z \mid l \le i \le k \text{ and } \varphi_{\mfn_m}(\mfm) \cge \ord_{\cle,\mfn_m}(g) \text{ for all } m\}$.

\subsection{Representing Algebraic Power Series}\label{sec.AlgSeries}

In the sequel we suggest a representation which is suitable for the computer algebra system {\tt MAGMA} \cite{MAGMA} and facilitates all of the above operations. We only highlight the essential points. {\tt MAGMA}, like most other computer algebra systems, does not provide polynomials  $f \in \E[\xu^{\Gamma}]$ with fractionary exponents directly. For simplicity of reading, we nevertheless use such polynomials, an implementation being straightforward. We also assume that given $f$ we can ask for its coefficient field $\E$ and exponent lattice $\Gamma$ and that an implicit conversion mechanism is provided for $\E_1[\xu^{\Gamma_1}] \subseteq \E_2[\xu^{\Gamma_2}]$ whenever $\E_1 \subseteq \E_2$ and $\Gamma_1 \subseteq \Gamma_2$. Further it should be understood that a tuple defining a series may recursively depend on other series or even on polynomials with series coefficients. We represent algebraic power series using a \emph{hybrid lazy-exact approach} by finite, acyclic, directed and rooted graphs with nodes of two types:

\begin{description}
\item[Type A] We represent an algebraic power series $\alpha \in \E_2\ll\xu^{\Gamma_2} \rr$ by a tuple
\begin{gather}\label{eqn.TypeA}
(\alpha_0, f)
\end{gather}
where $\alpha_0 \in \E_2[\xu^{\Gamma_2}]$ is an initial segment of $\alpha$ w.r.t.\ $\cle$ and $f = \sum_{0 \le i \le d} \varphi_iz^i \in \E_1\ll\xu^{\Gamma_1}\rr[z]$ are such that $\E_1 \subseteq \E_2$, $\Gamma_1 \subseteq \Gamma_2$ and $f \neq 0$ is squarefree (when considered as a polynomial over $\totfrac(\E_1\ll\xu^{\Gamma_1}\rr)$) and vanishing on $\alpha$.

Such a node in general has as $d+1$ descendants, namely, the coefficients $\varphi_i$. As a special case we allow $f \in \E_1[\xu^{\Gamma_1}][z]$. In this case we store $f$ itself, there are no descendants and the node is terminal.

We also need a technical condition to ensure that $\alpha_0$ identifies $\alpha$ uniquely amongst the roots of $f$.

\begin{algorithm}[H]
\caption{$\texttt{Series}(\alpha_0 : \E_2[\xu^{\Gamma_2}], f : \E_1\ll\xu^{\Gamma_1}\rr[z]) : \E_2\ll\xu^{\Gamma_2}\rr$}\label{alg.Series}
\begin{algorithmic}[1]
\REQUIRE A tuple as in \eqref{eqn.TypeA} fulfilling Condition~\ref{cnd.ValidTypeA}.
\ENSURE  The power series $\alpha$ defined by it.
\STATE Encapsulate data and return object of type ``algebraic power series'';
\end{algorithmic}
\end{algorithm}

\item[Type B] The second type of node is given by a tuple
\begin{gather}\label{eqn.TypeB}
(\alpha, \underline{\mfn}, \underline{\xi})
\end{gather}
and represents the image $\beta := \phi(\alpha)$ under the homomorphism described in the third requirement.  It has $l+1$ descendants, namely, the $\xi_i$ and $\alpha$. Also in this case we need a technical condition necessary for $\phi$ being well-defined.

\begin{algorithm}[H]
\caption{$\texttt{Evaluate}(\alpha : \E_1\ll\xu^{\Gamma_1}\rr, \underline{\mfn} : (\Gamma_1^{\vee} \cap \R_{\ge 0}^n)^l, \underline{\xi} : \E_2\ll\xu^{\Gamma_2}\rr) : \E_2\ll\xu^{\Gamma_2}\rr$}\label{alg.Evaluate}
\begin{algorithmic}[1]
\REQUIRE A tuple as in \eqref{eqn.TypeB} fulfilling Condition~\ref{cnd.ValidTypeB}.
\ENSURE  The power series $\beta$ defined by it.
\STATE Encapsulate data and return object of type ``algebraic power series'';
\end{algorithmic}
\end{algorithm}
\end{description}

Nodes of type B facilitate explicitely the operation of requirement (4) from our list. The following sections describe how to implement the remaining requirements using algorithms based on an idea in \cite{TB_JS:2005a} which is similar in spirit to \cite{JvdH:2006}.

\begin{rem}[Representation Paradigm]\label{rem.RepParadigm}
The reason for calling this representation \emph{lazy-exact} is the following: In order to expand a series up to some order, we will recursively expand its descendants (maybe storing precomputed values) and then compute the approximation of the series under consideration in a \emph{lazy} fashion. On the other hand, using elimination theory and traversing the representation graph recursively, one can compute a (minimal) defining polynomial for any series represented as above. Therefore zero-equality and polynomiality/rationality are decidable. In this sense we speak of an \emph{exact} representation. Since in particular all polynomials (including $0$ and $1$) are representable we get a \emph{computationally effective} ring. All this (and a little bit more) is included in the {\tt MAGMA} implementation. Note, however, that it is advisable to avoid the effective decision algorithms because they depend on nested resultant computations and can be very expensive.
\end{rem}

\subsection{Expanding Algebraic Power Series}\label{sec.Expand}

Now we show how to expand power series in a lazy fashion. For more efficient computations with lazy power series, see \cite{JvdH:2002}. For expanding a power series $\alpha$ given by a tuple \eqref{eqn.TypeA}, we consider
\begin{gather}\label{eqn.DefpolTrans}
g := f(z+\alpha_0) = \sum_{i=0}^d \gamma_i z^i, \alpha_1 := \alpha - \alpha_0, \mfn := \ord_{\cle}(\alpha_1) \text{ and } \mfc := \ord_{\cle}(\gamma_1).
\end{gather}

First we want to find implications of the fact that $g(\alpha_1)=0$. The case $\alpha_1 = 0$ happens if and only if $g(0)=0$ or, equivalently, $\gamma_0 = 0$. Let's assume $\alpha_1 \neq 0$ and set $\alpha' := \lt_{\cle}(\alpha_1)$. Next we study the contribution of the terms $\gamma_i z^i$ to the result under the substitution $z \mapsto \alpha_1$. Whenever $\gamma_i \neq 0$ we find \[\gamma_i \alpha_1^i =\gamma_i (\alpha' + (\alpha_1 - \alpha'))^i = \lt_{\cle}(\gamma_i) \alpha'^i + \dots \text{ (higher order terms)}.\] The minimal order of these expressions is \[\min_{\cle} \{\underbrace{\ord_{\cle}(\lt_{\cle}(\gamma_i) \alpha'^i)}_{\ord_{\cle}(\gamma_i) + i \mfn} \mid 0 \le i \le d \text{ and } \gamma_i \neq 0\} = \ord_{\cle,\mfn}(g).\] Since $g(\alpha_1) = 0$ and $\alpha' \neq 0$ it follows that the terms of order $\ord_{\cle,\mfn}(g)$ must sum up to zero. In particular, there have to be at least two indices $i$ s.t.\ the terms $\lt_{\cle}(\gamma_i) \alpha'^i$ are of order $\ord_{\cle,\mfn}(g)$. In other words, the $\mfn$-th edge equation of $g$ must be non-trivial and $\alpha'$ must be a root of it.

If we can make sure that the data in \eqref{eqn.TypeA} implies that the $\mfn$-th edge is linear of the form $\lt_{\cle}(\gamma_1) z + \lt_{\cle}(\gamma_0) = 0$ then $\alpha'$ is uniquely determined and easy to compute.

\begin{cnd}[Valid Representations of Type A]\label{cnd.ValidTypeA}
With definitions as in \eqref{eqn.DefpolTrans} we require for valid representations of type A that either $\gamma_0 = 0$ or $\edge_{\cle,\mfn}(g)$ is linear.
\end{cnd}

Now substitute $\alpha_0 \mapsto \alpha_0 + \alpha'$, $\alpha_1 \mapsto \alpha_1 - \alpha'$ and $g \mapsto g(z+\alpha')$. A careful analysis of the exponent structure shows that under the above condition $\lt_{\cle}(\gamma_1)$ doesn't change and Condition~\ref{cnd.ValidTypeA} remains valid. Hence we can repeat the argument and compute successively as many terms as needed.

\begin{ex}\label{ex.Expand}
Let $f := z^6 - 3x_2z^4 - \tfrac{1}{64}x_1^2x_2^3z^3 + 3x_2^2z^2 - x_2^3$. Then $f$ has a power series root $\alpha$ starting with $\alpha_0 = x_2^{1/2} + \tfrac{1}{8} x_1^{2/3} x_2$. We work with the degree graded lexicographical ordering $\cle$. First we translate $f$ as in \eqref{eqn.DefpolTrans}:
\begin{align*}
f(z + \alpha_0) = & z^6 +
(- 6x_2^{1/2} + \tfrac{3}{4}x_1^{2/3}x_2)z^5 + \\
&(+ 12x_2 - \tfrac{15}{4}x_1^{2/3}x_2^{3/2} + \tfrac{15}{64}x_1^{4/3}x_2^{2})z^4 + \\
&(- 8x_2^{3/2} + 6x_1^{2/3}x_2^{2} - \tfrac{15}{16}x_1^{4/3}x_2^{5/2} + \tfrac{3}{128}x_1^{2}x_2^{3})z^3 + \\
&(- 3x_1^{2/3}x_2^{5/2} + \tfrac{9}{8}x_1^{4/3}x_2^{3} - \tfrac{9}{128}x_1^{2}x_2^{7/2} - \tfrac{9}{4096}x_1^{8/3}x_2^{4})z^2 + \\
&(- \tfrac{3}{8}x_1^{4/3}x_2^{7/2} + \tfrac{3}{64}x_1^{2}x_2^{4} + \tfrac{9}{2048}x_1^{8/3}x_2^{9/2} - \tfrac{9}{16384}x_1^{10/3}x_2^{5})z + \\
&(- \tfrac{3}{1024}x_1^{8/3}x_2^{5} + \tfrac{9}{16384}x_1^{10/3}x_2^{11/2} - \tfrac{7}{262144}x_1^{4}x_2^{6})
\end{align*}
We find that $f(z + \alpha_0)$ has a linear edge of slope $(\tfrac{4}{3},\tfrac{3}{2})$ and extract the edge equation $- \tfrac{3}{8}x_1^{4/3}x_2^{7/2}z - \tfrac{3}{1024}x_1^{8/3}x_2^{5}  = 0$. The solution is the next term $- \tfrac{1}{128}x_1^{4/3}x_2^{3/2}$ in the expansion of $\alpha$. Next we consider $f(z + \alpha_0 - \tfrac{1}{128}x_1^{4/3}x_2^{3/2})$ and find that it has a linear edge of slope $(\tfrac{8}{3},\tfrac{5}{2})$. Solving the edge equation $- \tfrac{3}{8}x_1^{4/3}x_2^{7/2} + \tfrac{3}{262144}x_1^{4}x_2^{6}=0$ we find the next term $\tfrac{1}{32768}x_1^{8/3}x_2^{5/2}$, and so on: \[\alpha = - x_2^{1/2} + \tfrac{1}{8}x_1^{2/3}x_2 - \tfrac{1}{128}x_1^{4/3}x_2^{3/2} + \tfrac{1}{32768}x_1^{8/3}x_2^{5/2} - \tfrac{1}{4194304}x_1^{4}x_2^{7/2} + \dots\]
\end{ex}

Obviously we do not need to know $g$ completely to do this computation. More precisely, we can expand $\alpha$ up to order less than $\mfo$, if we have approximated $g$ sufficiently well. All the terms of order less than $\mfo$ are determined by a linear edge equation of the form $\lt_{\cle}(\gamma_1)z+\dots = 0$. Therefore, it is sufficient to know the constant term up to order less than $\ord_{\cle,\mfo}(\lt_{\cle}(\gamma_1) z) = \mfc + \mfo$ with $\mfc$ as in \eqref{eqn.DefpolTrans}. Hence, for approximating $g$ it is sufficient to expand the coefficients of $f$ up to order less than $\mfc + \mfo$.

For example, Condition~\ref{cnd.ValidTypeA} is true when considering a power series defined by the implicit function theorem. Therefore solving such equations is trivial:

\begin{algorithm}[H]
\caption{$\texttt{ImplicitFunction}(g : \E_1\ll\xu^{\Gamma_1}\rr[z]) : \E_1\ll\xu^{\Gamma_1}\rr$}\label{alg.ImplicitFunction}
\begin{algorithmic}[1]
\REQUIRE A polynomial $g$ s.t.\ $g(0,\dots,0,0) = 0$ and $\partial g/\partial z(0,\dots,0,0) \neq 0$.
\ENSURE  The unique root $\alpha$ of $g$ s.t.\ $\alpha(0,\dots,0)=0$.
\STATE \return\ $\texttt{Series}(0, g)$;
\end{algorithmic}
\end{algorithm}

Indeed, if $g$ is equal to $\sum_{i=0}^d \gamma_i z^i$, $g(0,\dots,0,0) = 0$ and $\partial g/\partial z(0,\dots,0,0) \neq 0$ then $\ord_{\cle}(\gamma_0) \cgt 0$ (or $\gamma_0 = 0$) and $\ord_{\cle}(\gamma_1) = 0$. If $\gamma_0 \neq 0$ then $\edge_{\cle,\mfn}(g)$ must be linear where $\mfn = \ord_{\cle}(\gamma_0) = \ord_{\cle}(\alpha)$.

\begin{rem}[Universality of Type A]\label{rem.UniTypeA}
Note that for any algebraic power series $\alpha$ there is, by definition, a squarefree polynomial $f$ with polynomial coefficients vanishing on it. This implies that $f(z+\alpha)$ vanishes at $z=0$ with multiplicity one or that the constant coefficient of $f(z+\alpha)$ is zero whereas its linear coefficient does not vanish. This means that the initial term of the linear coefficient is fixed if we consider translations $f(z+\alpha_0)$ by sufficiently large initial segments $\alpha_0$ of $\alpha$. As a consequence we can always find an initial segment $\alpha_0$ s.t.\ Condition~\ref{cnd.ValidTypeA} is fulfilled and every algebraic power series is representable by a single node of type A.

So our representation is already universal. But computing efficiently the last operation of our requirement list is not easy in this representation. So for algorithmic purposes we also allow the second type of node.
\end{rem}

The previous discussion showed how to expand power series represented by a graph consisting of nodes of type A. Since $\cle$ is compatible with total degree this implies that we can expand such series up to arbitrary total degree. Now we consider a power series $\beta$ represented by a tuple \eqref{eqn.TypeB}. Indeed what comes next is best explained in the total degree ordering.

Assume we want to expand $\beta$ up to order less than $o \in \Q$. Therefore we first compute approximations $\widetilde{\xi}_i$ of the descendants $\xi_i$ up to order less than $o$. Then we apply the map $\xu^{\mfm} \mapsto \prod_{1 \le i \le l} \widetilde{\xi}_i^{\mfn_i(\mfm)}$ to each term in an expansion of $\alpha$, sum up the intermediate results and finally truncate at order $o$. For this truncation to be correct, we have to use a sufficiently large expansion of $\alpha$, say, up to order less than $o'$. The only remaining question is how to determine $o'$ from $o$.

To this end let
\begin{gather}\label{eqn.MultContraction}
\widetilde{\mfn} := \left( \sum_{1 \le i \le l} \ord_{\le}(\xi_i)~\mfn_i \right) = (\widetilde{\mfn}_1,\dots,\widetilde{\mfn}_n) \text{ and } c := \min_{\le}\{\widetilde{\mfn}_j \mid 1 \le j \le n\}
\end{gather}
where $\ord_{\le}(0) := \infty$. Let $\phi$ denote the homomorphism $\xu^{\mfm} \mapsto \prod_{1 \le i \le l} \xi_i^{\mfn_i(\mfm)}$ then
\begin{multline*}
\ord_{\le}(\phi(\xu^{\mfm})) = \ord_{\le} \left(\prod_{1 \le i \le l} \xi_i^{\mfn_i(\mfm)}\right) = \left( \sum_{1 \le i \le l} \ord_{\le}(\xi_i)~\mfn_i(\mfm) \right) = \widetilde{\mfn}(\mfm) \\
= \sum_{1 \le j \le n} \widetilde{\mfn}_j \mfm_j \ge \sum_{1 \le j \le n} c \mfm_j = c \lvert \mfm \rvert.
\end{multline*}
This calculation shows two things: First, if $c > 0$ then for any $\gamma = \sum_{\mfm} \gamma_\mfm \xu^{\mfm}$ the sum $\sum_{\mfm} \gamma_\mfm \phi(\xu^{\mfm})$ converges, so $\phi(\gamma)$ is well defined. This is an analogue to the usual condition when substituting formal power series into each other.

\begin{cnd}[Valid Representations of Type B]\label{cnd.ValidTypeB}
For valid representations of type B we require $c>0$ for $c$ defined as in \eqref{eqn.MultContraction}.
\end{cnd}

Second, under this condition, if $\alpha = \sum_{\mfm} \alpha_\mfm \xu^{\mfm}$ then the terms $\alpha_\mfm \xu^{\mfm}$ with $\lvert \mfm \rvert \ge o/c$ contribute terms to $\phi(\alpha)$ of order greater or equal $o$, so we can choose $o' := o/c$.

\begin{rem}[Contraction Constants]
We can give a nice theoretical meaning to the two technical conditions and the deduced algorithms. Namely, the values $\mfc$ and $c$ may be understood as additive resp.\ multiplicative \emph{contraction constants} for certain continuous maps between power series domains with the usual metrics. In the case of Condition~\ref{eqn.TypeA} this map is contractive only in a small enough neighborhood (determined by $\ord_{\cle}(\alpha-\alpha_0)$) of the root $\alpha$. Contractivity makes the represented series well-defined.

In both cases the constants determine how far the descendants of a node have to be expanded. An implementation could store them together with the representing tuples \eqref{eqn.TypeA} and \eqref{eqn.TypeB}. We would also like to mention that for an efficient implementation it is \emph{crucial to truncate intermediate results} (e.g.\ defining polynomials or approximations of descendants) as often as possible to a sufficient precision. The bounds are computed by similar calculations as above. We omit them from the presentation and show only the essentials of the algorithm.
\end{rem}

\subsection{Rational Parametrizations for Quasi-Ordinary Polynomials}\label{sec.RatParams}

A quite intricate thing is the computation of complete sets of rational parametrizations. In requirement (3) we have introduced the action of the group $\hom(\Gamma,\E^*)$ on $\E\ll\xu^{\Gamma}\rr$. We write this as a left action using again functional notation, {\it i.e.}, for $\sigma \in \hom(\Gamma,\E^*)$ and $\alpha \in \E\ll\xu^{\Gamma}\rr$ we simply write $\sigma(\alpha)$. This is an \emph{exponent structure preserving} automorphism in the sense that $\supp_{\cle}(\alpha) = \supp_{\cle}(\sigma(\alpha))$. Let's convince ourselves that this action can effectively be carried out in our representation.

\begin{rem}[Effectivity of Actions Induced by Lattice Homomorphisms]
Since such an automorphism is structure preserving, it is easily computed in our representation: If a power series is represented by a tuple \eqref{eqn.TypeA}, then we apply the automorphism to the initial segment $\alpha_0$ and to the coefficients of the defining equation $f$, possibly recursing to the descendants. If a power series is represented by a tuple \eqref{eqn.TypeB}, we apply the automorphism recursively to the $\xi_i$.
\end{rem}

\begin{dfn}[Rational Parametrizations]\label{dfn.RatParFine}
We call $(\sigma, \alpha)$ with $\alpha \in \E_1\ll\xu^{\Gamma_1}\rr$ and $\sigma \in \hom(\Gamma_0,\E_1^*)$ a \emph{parametrization} of a monic polynomial $f \in \E_0\ll\xu^{\Gamma_0}\rr[z]$ iff $\E_0 \subseteq \E_1$, $\Gamma_0 \subseteq \Gamma_1$ and $\sigma\lift{z}(f)(\alpha)=0$. The \emph{order} of the parametrization is defined to be $\ord_{\cle}(\alpha)$.

Let $g \vert f$ be an irreducible factor s.t.\ $\sigma\lift{z}(g)(\alpha)=0$. We call $(\sigma, \alpha)$ \emph{rational} if the induced homomorphism $\ic(\E_0\ll\xu^{\Gamma_0}\rr[z]/\langle g \rangle) \to \E_1\ll\xu^{\Gamma_1}\rr$ which maps $z \mapsto \alpha$ and $\gamma \mapsto \sigma(\gamma)$ for $\gamma \in \E_0\ll\xu^{\Gamma_0}\rr$ is an isomorphism.
\end{dfn}

Intuitively, rational parametrizations are distinguished by minimal field and lattice extensions. The induced homomorphism exists due to the universal property of integrally closed domains. If this homomorphism is an isomorphism then the irreducible polynomial $g$ from above at least has to be prime. (Mind that $\E_0\ll\xu^{\Gamma_0}\rr$ is in general no UFD!)

\begin{rem}[Irreducible Monic Polynomials]
Let $g \in \E_0\ll\xu^{\Gamma_0}\rr[z]$ be a monic irreducible polynomial. We want to prove that $g$ is prime. The following arguments are taken from \cite{SMcA:2001}. It is sufficient to show primality in $\totfrac(\E_0\ll\xu^{\Gamma_0}\rr)[z]$; Indeed, if $hg \in \E_0\ll\xu^{\Gamma_0}\rr[z]$ for some $h \in \totfrac(\E_0\ll\xu^{\Gamma_0}\rr)[z]$ then a Gaussian style inductive argument shows that $h \in \E_0\ll\xu^{\Gamma_0}\rr[z]$. So $g~\E_0\ll\xu^{\Gamma_0}\rr[z] = (g~\totfrac(\E_0\ll\xu^{\Gamma_0}\rr)[z]) \cap \E_0\ll\xu^{\Gamma_0}\rr[z]$, {\it i.e.}, it is the preimage of a prime ideal and therefore prime itself.

Let $g' \in \totfrac(\E_0\ll\xu^{\Gamma_0}\rr)[z]$ be an irreducible (and hence prime) monic polynomial factor of $g$. We show $g'=g$. Being also roots of $g$, all roots of $g'$ (in some splitting field) are integral over $\E_0\ll\xu^{\Gamma_0}\rr$. The coefficients of $g'$ (being polynomials in these roots) are also integral over $\E_0\ll\xu^{\Gamma_0}\rr$ and elements of $\totfrac(\E_0\ll\xu^{\Gamma_0}\rr)$. Then $g' \in \E_0\ll\xu^{\Gamma_0}\rr[z]$ because $\E_0\ll\xu^{\Gamma_0}\rr$ is integrally closed. By the above argument $g' \vert g$ also in $\E_0\ll\xu^{\Gamma_0}\rr[z]$ and so $g'=g$ because $g$ is irreducible.

This also implies that a monic polynomial $f \in \E_0\ll\xu^{\Gamma_0}\rr[z]$ has a unique (up to permutation) factorization into monic irreducibles.
\end{rem}

Let now $f \in \E_0\ll\xu^{\Gamma_0}\rr[z]$ be quasi-ordinary and assume we want to compute a \emph{complete set of rational parametrizations}. Let's be more general and say we want to compute a complete set of rational parametrizations of order greater than some value $\mfn_0 \in \Q^n$. (With the choice $\mfn_0 := (-1,\dots,-1)$ this includes the original task.)

First assume we are given a parametrization $(\sigma, \alpha)$ of $f$ s.t.\ $\alpha \neq 0$ and $\mfn := \ord_{\cle}(\alpha) \cgt \mfn_0$. Write $g := \sigma\lift{z}(f)$ and $\alpha' := \lt_{\cle}(\alpha)$. By the same reasoning as in the beginning of Section~\ref{sec.Expand} (with $\alpha$ in place of $\alpha_1$), we find that the $\mfn$-th edge equation of $g$ must be non-trivial and $\alpha'$ must be a root of it.

Now assume we have to find $(\sigma, \alpha)$ as above using a recursive approach. As a first step we have to determine the initial term of $\alpha$ (up to some isomorphism). From the previous discussion it follows that its slope will be $\mfn \cgt \mfn_0$ s.t.\ $f$ has a non-trivial $\mfn$-th edge equation. If the tuple $(\sigma, \alpha)$ is meant to be a rational parametrization, then we have to solve this equation economically, {\it i.e.}, using a field extension of least degree, and at the same time determine $\sigma$ partially. Duval's trick \cite{DD:1989} adapted to the multivariate case gives an optimal choice.

To this end let $\mfm_i$ for $1 \le i \le n$ be a basis of $\Gamma_0$. Let $b \in \Z_{>0}$ be the size of $\Gamma' / \Gamma_0$ where $\Gamma' := \Gamma_0 + \Z \mfn$, {\it i.e.}, $b$ is minimal s.t.\ $b \mfn \in \Gamma_0$. Now we can write \[\edge_{\cle,\mfn}(f) = \sum_{l \le i \le k} f_{i} \xu^{\mfm_i}z^i = \xu^{\mfm_l}z^l \sum_{0 \le j \le (k-l)/b} f_{l + jb} \xu^{-jb\mfn}z^{jb}\] with $f_l \neq 0$ and $f_k \neq 0$. Further we can express $b \mfn = \sum_{1 \le i \le n} c_i \mfm_i$, where necessarily $\gcd(b,c_1,\dots,c_n) = 1$ since $b$ was chosen minimally, and compute B{\'e}zout coefficients $u$ and $v_i$ s.t.\ $ub + \sum_{1 \le i \le n} v_i c_i = 1$. (Note that this choice is not unique.) If $r$ is a \emph{non-zero} root in a minimal field extension $\E'$ s.t.\ \[\sum_{0 \le j \le (k-l)/b} f_{l + jb} r^j = 0\] then we can define the homomorphism $\sigma': \Gamma_0 \to \E'^*: \mfm_i \mapsto r^{-v_i}$ and the initial term $\alpha' := r^u \xu^{\mfn}$. With these definitions one checks
\begin{align*}
\edge_{\cle,\mfn}(\sigma'\lift{z}(f))(\alpha') & = \sigma'\lift{z}(\edge_{\cle,\mfn}(f))(r^u \xu^{\mfn}) \\
& = \sigma'(\xu^{\mfm_l})z^l \sum_{0 \le j \le (k-l)/b} f_{l + jb} \sigma'(\xu^{-jb\mfn})(r^u \xu^{\mfn})^{jb} \\
& = (\dots) \sum_{0 \le j \le (k-l)/b} f_{l + jb} \left(\sigma'(\xu^{-\sum_{1 \le i \le n} c_i \mfm_i})r^{ub} \xu^{b\mfn}\right)^{j} \\
& = (\dots) \sum_{0 \le j \le (k-l)/b} f_{l + jb} \left(r^{\sum_{1 \le i \le n} v_i c_i} \xu^{-b\mfn} r^{ub} \xu^{b\mfn}\right)^{j}
\end{align*}
\begin{align*}
\text{ \hspace{1.85cm} }& = (\dots) \sum_{0 \le j \le (k-l)/b} f_{l + jb} \left(r^{ub + \sum_{1 \le i \le n} v_i c_i} \xu^{-b\mfn + b\mfn}\right)^{j} \\
& = (\dots) \sum_{0 \le j \le (k-l)/b} f_{l + jb} r^j = 0.
\end{align*}

This finishes the description of how to choose the first term (up to isomorphism) and part of the structure preserving automorphism. The remainder of course has to be chosen via recursion. More precisely, define $f' := \sigma'\lift{z}(f)(z + \alpha')$ and compute a parametrization $(\sigma'', \alpha'') \in \hom(\Gamma',\E_1^*) \times \E_1\ll\xu^{\Gamma_1}\rr$ of $f'$ with $\ord_{\cle}(\alpha'') \cgt \mfn$. Then we may just set $\alpha := \sigma''(\alpha') + \alpha'' \in \E_1\ll\xu^{\Gamma_1}\rr$ and $\sigma := \sigma'' \sigma' \in \hom(\Gamma_0,\E_1^*)$. (Note that we have an inclusion $\hom(\Gamma_0,\E'^*) \to \hom(\Gamma_0,\E_1^*)$ and a surjection $\hom(\Gamma',\E_1^*) \to \hom(\Gamma_0,\E_1^*)$, hence, we can build $\sigma'' \sigma' \in \hom(\Gamma_0,\E_1^*)$.) Now trivially
\begin{multline*}
0 = \sigma''\lift{z}(f')(\alpha'') = \sigma''\lift{z}(\sigma'\lift{z}(f)(z + \alpha'))(\alpha'') \\
= \sigma''\lift{z}(\sigma'\lift{z}(f))(z + \sigma''(\alpha'))(\alpha'') = \sigma''\lift{z}(\sigma'\lift{z}(f))(\sigma''(\alpha') + \alpha'') = \sigma\lift{z}(f)(\alpha).
\end{multline*}

\begin{ex}
Let $f_0 := z^6 + 3x_2z^4 + x_1^2x_2^3z^3 + 3x_2^2z^2 + x_2^3$. Its discriminant is $\disc_z(f_0) = 729 x_1^8x_2^{21}(x_1^4x_2^3 - 64)$ and, hence, it is quasi-ordinary. We again use the degree graded lexicographical ordering $\cle$.

The only rational slope of a non-trivial edge is $\mfn_1 := (0, \tfrac{1}{2})$ with edge equation $z^6 + 3x_2z^4 + 3x_2^2z^2 + x_2^3 = 0$. We see that only even powers of $z$ have a non-vanishing coefficient which corresponds to the fact that $2\mfn_1 \in \Z^2$, {\it i.e.}, $b_1=2$. Let $r_1$ be a solution of $0 = r^3 + 3r^2 + 3r + 1 = (r + 1)^3$, hence, $r_1=-1$.
The standard lattice has basis $\mfm_{1,1} = (1,0)$ and $\mfm_{1,2} = (0,1)$. Then $2 \mfn_{1,1} = 0 \mfm_{1,1} + 1 \mfm_{1,2}$, hence, $c_{1,1} = 0$ and $c_{1,2} = 1$. One verifies the B{\'e}zout relation $0b_1 + 0c_{1,1} + 1c_{1,2} = 1$ and we have $u_1=0$, $v_{1,1}=0$ and $v_{1,2}=1$.
Therefore we define the first term $\alpha_1 := r_1^{u_1} x_2^{1/2} = x_2^{1/2}$ and the homomorphism to be $\sigma_1: \mfm_{1,1} \mapsto r_1^{-v_{1,1}} = 1, \mfm_{1,2} \mapsto r_1^{-v_{1,2}} = -1$ with corresponding action $x_1 \mapsto x_1, x_2 \mapsto -x_2$. We set
\begin{multline*}
f_1 := \sigma_1\lift{z}(f_0)(z + x_2^{\tfrac{1}{2}}) = \\ z^6 + 6x_2^{\tfrac{1}{2}}z^5 + 12x_2z^4 + (8x_2^{\tfrac{3}{2}} - x_1^{2}x_2^{3})z^3 - 3x_1^{2}x_2^{\tfrac{7}{2}}z^2 - 3x_1^{2}x_2^{\tfrac{8}{2}}z - x_1^{2}x_2^{\tfrac{9}{2}}.
\end{multline*}

For $f_1$ we now find two rational slopes of non-trivial edges, namely, $\mfn_1$ from above and $\mfn_2 = (\tfrac{2}{3}, 1)$. The extended exponent lattice at this point is $\Z \mfm_{2,1} + \Z \mfm_{2,2}$ where $\mfm_{2,1} = (1,0)$ and $\mfm_{2,2} = (0, \tfrac{1}{2})$. We have $3 \mfn_2 = 2 \mfm_{2,1} + 6 \mfm_{2,2}$, {\it i.e.}, $b_2 = 3$, $c_{2,1} = 2$ and $c_{2,2} = 6$. From a B{\'e}zout relation we get $u_2 = 1$, $v_{2,1} = -1$ and $v_{2,2} = 0$. The edge equation is $8x_2^{3/2}z^3 - x_1^{2}x_2^{9/2} = 0$. Let $r_2$ be the root of $8r-1=0$, {\it i.e.}, $r_2=\tfrac{1}{8}$. Hence, we define a new homomorphism $\sigma_2:  \mfm_{2,1} \mapsto r_2^{-v_{2,1}} = \tfrac{1}{8}, \mfm_{2,2} \mapsto r_2^{-v_{2,2}} = 1$ with action $x_1 \mapsto \tfrac{1}{8}x_1, x_2^{1/2} \mapsto x_2^{1/2}$. We also define a new initial segment $\alpha_2 := \sigma_2(\alpha_1) + r_2^{u_2} x_1^{2/3} x_2 = x_2^{1/2} - \tfrac{1}{8} x_1^{2/3} x_2$ and set $f_2 := \sigma_2\lift{z}(f_1)(z - \tfrac{1}{8} x_1^{2/3} x_2)$.

From now on the edge equation will always be linear and we have determined a parametrization. More precisely, the overall automorphism is $\sigma_2\sigma_1$ and acts by $x_1 \mapsto \tfrac{1}{8}x_1, x_2 \mapsto -x_2$. The power series starts with $\alpha_2$ and is a root of $\sigma_2\sigma_1\lift{z}(f_0)$. We have seen its expansion in Example~\ref{ex.Expand}.
\end{ex}

This discussion yields the following algorithm. The return type $\overline{\E_0}\ll\xu^{\Q^n} \rr$ means algebraic power series with coefficients in some finite algebraic extension $\E_0$ and exponents in some rational lattice containing $\Z^n$.

\begin{algorithm}[H]
\caption{$\texttt{ParamRec}(f : \E_0\ll\xu^{\Gamma_0}\rr[z], \mfn_0 \in \Q^n) : 2^{\hom(\Gamma_0, \overline{\E_0}^*) \times \overline{\E_0}\ll\xu^{\Q^n} \rr}$}\label{alg.ParamRec}
\begin{algorithmic}[1]
\REQUIRE A quasi-ordinary polynomial $f$ and an order $\mfn_0$.
\ENSURE  A set of rational parametrizations of $f$ of order greater $\mfn_0$.
\STATE $S := \{\mfn \in \Q_{\ge 0}^n \mid \mfn \cgt \mfn_0 \text{ and } \edge_{\cle,\mfn}(f) \text{ is non-trivial}\}$; $\PP := \emptyset$;
\IF{$(\forall \mfn \in S:  \ord_z(\edge_{\cle,\mfn}(f)) \ge 1)$ or $(\exists \mfn \in S: \deg_z(\edge_{\cle,\mfn}(f)) = 1)$}
\STATE $S := S \setminus \{\mfn \in S \mid \deg_z(\edge_{\cle,\mfn}(f)) = 1\}$; $\PP := \PP \cup \{(1, \texttt{Series}(0,f))\}$;
\ENDIF
\FOR{$\mfn \in S$}
\STATE Let $b \in \Z_{>0}$ be minimal s.t.\ $b \mfn \in \Gamma_0$; Let $\mfm_i$ for $1 \le i \le n$ be a basis of $\Gamma_0$;
\STATE Write $b \mfn = \sum_{1 \le i \le n} c_i \mfm_i$ and compute $u, v_i \in \Z$ s.t.\ $ub + \sum_{1 \le i \le n} v_i c_i = 1$;
\STATE Write $\edge_{\cle,\mfn}(f) = \xu^{\mfm_l}z^l \sum_{0 \le j \le k} f_{j} \xu^{-jb\mfn}z^{jb}$ with $f_0 \neq 0$ and $f_k \neq 0$;
\STATE $R := \texttt{ZeroSet}(\langle \sum_{0 \le j \le k} f_{j} z^{j} \rangle)$;
\FOR{$r \in R$}
\STATE $\alpha' := r^u \xu^{\mfn}$; $\sigma' : \Gamma_0 \to \E'^* : \mfm_i \mapsto r^{-v_i}$;
\COMMENT{assuming $r \in \E'^*$}
\STATE $\PP := \PP \cup \map_{(\sigma'', \alpha'') \mapsto (\sigma''\sigma', \sigma''(\alpha')+\alpha'')} \texttt{ParamRec}(\sigma'\lift{z}(f)(z + \alpha'), \mfn)$;
\ENDFOR
\ENDFOR
\STATE \return\ $\PP$;
\end{algorithmic}
\end{algorithm}

For the definition of the auxiliary function \texttt{ZeroSet} see Remark~\ref{rem.AuxFunc}. Applying structure preserving automorphisms and translating the defining polynomial in $z$ preserves quasi-ordinariness. Therefore the arguments to the recursive call always fit to the input description. Although the algorithm never makes explicit use of Theorem~\ref{thm.JungAbhyank} (the Theorem of Jung-Abhyankar) it depends crucially on $f$ being quasi-ordinary. For example, the power series constructed in line $3$ are well-defined in that case; Otherwise there might not be any power series root of $f$ which is supported on the positive orthant and whose initial term is a root of the linear edge of $f$. Recursion in this algorithm ends when the set $S$ is empty when entering line $4$. Termination is assured by the very argument of Remark~\ref{rem.UniTypeA}.

\begin{rem}[And a Little Bit of Engineering\dots]\label{rem.Engineering}
In line $1$ of the above algorithm we compute the ``Newton Polygon'' of $f = \sum_i \varphi_i z^i$. Since $f$ is monic, one of the non-trivial edges contains the point $((0,\dots,0),\deg_z(f))$. Now $z$ divides $f$ either with multiplicity $0$ or $1$ (because of squarefreeness). So for computing the non-trivial edges it will be enough to expand the coefficients $\varphi_i$ until the initial term of $\varphi_1$ or $\varphi_0$ appears depending on whether $z$ divides $f$ or not. In other words we would have to check, whether $\varphi_0 = 0$. In principle this is possible, see Remark~\ref{rem.RepParadigm}, but from the point of efficiency it is not advisable.

On the other hand it doesn't matter if $z$ divides $f$ or if $f$ has a linear edge equation (compare to the condition in line $2$). In both cases we return $\texttt{Series}(0,f)$ which might be zero. So an engineering solution might be to compute either all non-trivial edges or only a set of non-trivial edges with lowest vertex at linear level and make sure that the constant term (if existent) has large enough order.

Another approach would be to modify the above algorithm to allow approximate input with polynomial coefficients only but also include error reporting when an approximation was not accurate enough.

Further it is advisable to choose a short vector $(u,v_1,\dots,v_n) \in \Z^{n+1}$ using for example the LLL-algorithm. For practical purposes the additional complexity is negligible compared to the grow of coefficients if a larger vector is chosen.

We didn't include the technical tricks into the description because the algorithm is easier to read and to argue about when written up as above.
\end{rem}

Setting now $\mfn_0 := (-1,\dots,-1)$ we get a set of parametrizations of $f$:

\begin{algorithm}[H]
\caption{$\texttt{Param}(f : \E_0\ll\xu^{\Gamma_0}\rr[z]) : 2^{\hom(\Gamma_0, \overline{\E_0}^*) \times \overline{\E_0}\ll\xu^{\Q^n} \rr}$}\label{alg.Param}
\begin{algorithmic}[1]
\REQUIRE A quasi-ordinary polynomial $f$.
\ENSURE  A complete set of rational parametrizations of $f$.
\STATE \return\ $\texttt{ParamRec}(f, (-1,\dots,-1))$;
\end{algorithmic}
\end{algorithm}

We have to show that the computed parametrizations are rational and complete.

\begin{lem}[Bounding Extensions]\label{lem.BoundExts}
Let Algorithm~\ref{alg.ParamRec} be called with a quasi-ordinary polynomial $f \in \E_0\ll\xu^{\Gamma_0}\rr[z]$ and $\mfn_0 \in \Q^n$ and write $(\sigma_i, \alpha_i) \in \hom(\Gamma_0,\E_i^*) \times \E_i\ll\xu^{\Gamma_i}\rr$ for the computed parametrizations. Then we may bound the extensions from above by \[\sum_{i} [\E_i : \E_0]~\#(\Gamma_i / \Gamma_0) \le \max \{\deg_z(\edge_{\cle,\mfn}(f)) \mid \mfn \cgt \mfn_0\}.\]
\end{lem}
\begin{proof}
Assume the condition in line $2$ holds. Then either $f$ has a zero root, {\it i.e.}, $z \vert f$ or we have found a linear edge equation. In both cases we construct one parametrization of $f$ as in line $3$. The involved series is an element of $\E_0\ll\xu^{\Gamma_0}\rr$ and the lattice remains unchanged. This parametrization contributes a summand $1$ to the lefthand side of the inequality.

Now set $d := \max \{\deg_z(\edge_{\cle,\mfn}(f)) \mid \mfn \cgt \mfn_0\}-1$ if the condition in line $2$ was true, $d := \max \{\deg_z(\edge_{\cle,\mfn}(f)) \mid \mfn \cgt \mfn_0\}$ otherwise. We have to ensure that the sum modified by running only over the parametrizations constructed in the loop is bounded from above by $d$. To this end assume that $S$ contains $o$ slopes $\mfn_j$ when entering line $4$. Let $\ell_j$ denote the height of the edge equation of slope $\mfn_j$, {\it i.e.}, the difference of its $z$-degree and its $z$-order. Then $\sum_{1 \le j \le o} \ell_j \le d$.

Fix an edge, {\it i.e.}, a value of $j$. Then we compute an integer $b_j$ in line $5$ and a number of roots $r_{j,k} \in \E_{j,k}$ with multiplicities, say, $m_{j,k}$ for $1 \le k \le s_j$ in line $8$. It follows that $b_j = \#(\Gamma_j/\Gamma_0)$ with $\Gamma_j := \Gamma_0 + \Z \mfn_j$ measures the extension of the exponent lattice and $\ell_j = b_j (\sum_{1 \le k \le s_j} [\E_{j,k} : \E_0]~m_{j,k}$).

Fix a root, {\it i.e.}, a value of $k$ and let $f_{j,k}$ be the polynomial used as parameter in the recursive call in line $11$. Now $m_{j,k}$ is also the multiplicity of the root of the corresponding edge equation. After translation zero becomes an $m_{j,k}$-fold root, hence, the $z$-order of $\edge_{\cle,\mfn_j}(f_{j,k})$ is equal to $m_{j,k}$. This can serve as an upper bound in the statement for the recursive call, which returns, say, $t_{j,k}$ different parametrizations.

Assuming that the lemma is true for recursive calls (see $(*)$ below) one computes:
\begin{align*}
d & \ge \sum_{1 \le j \le o} \ell_j = \sum_{1 \le j \le o} b_j \sum_{1 \le k \le s_j} [\E_{j,k} : \E_0]~m_{j,k} \\
& \stackrel{(*)}{\ge} \sum_{1 \le j \le o} \#(\Gamma_j/\Gamma_0) \sum_{1 \le k \le s_j} [\E_{j,k} : \E_0] \sum_{1 \le l \le t_{j,k}} [\E_{j,k,l} : \E_{j,k}]~\#(\Gamma_{j,k,l}/\Gamma_{j}) \\
& = \sum_{1 \le j \le o} \sum_{1 \le k \le s_j} \sum_{1 \le l \le t_{j,k}} [\E_{j,k,l} : \E_{j,k}][\E_{j,k} : \E_0]~\#(\Gamma_{j,k,l}/\Gamma_{j})\#(\Gamma_j/\Gamma_0) \\
& = \sum_{1 \le j \le o} \sum_{1 \le k \le s_j} \sum_{1 \le l \le t_{j,k}} [\E_{j,k,l} : \E_0]~\#(\Gamma_{j,k,l}/\Gamma_0)
\end{align*}
For each tuple of indices $(j,k,l)$ appearing in that sum, the loop now produces exactly one parametrization with coefficient field $\E_{j,k,l}$ and exponent lattice $\Gamma_{j,k,l}$.
\end{proof}

For what follows we assume that $\E_0 \subseteq \E_i \subseteq \overline{\E_0}$. A parametrization $(\sigma_i, \alpha_i) \in \hom(\Gamma_0,\E_i^*) \times \E_i\ll\xu^{\Gamma_i}\rr$ is very close to a root of $f$ in $\overline{\E_0}\ll\xu^{\Gamma_i}\rr$. It provided a root if we could reverse the effects of $\sigma_i$. Precisely, let $\tau \in \hom(\Gamma_i, \overline{\E_0})$ be an extension of the inverse of $\sigma_i$ to $\Gamma_i$, in other words, $\tau\vert_{\Gamma_0} = \sigma_i^{-1}$. Then we call $\beta := \tau(\alpha_i)$ an \emph{associated root} to $(\sigma_i, \alpha_i)$, because applying $\tau$ to the equation $0 = \sigma_{i}\lift{z}(f)(\alpha_i)$ yields $0 = (\tau \sigma_{i})\lift{z}(f)(\tau(\alpha_i)) = (\sigma_{i}^{-1} \sigma_{i})\lift{z}(f)(\beta) = f(\beta)$. It is not hard to show that $\sigma_i^{-1}$ can always be extended to $\hom(\Gamma_i, \overline{\E_0})$ because $\overline{\E_0}$ is algebraically closed.

The next lemma says, that we do not miss any associated roots.

\begin{lem}[Completeness]\label{lem.Completeness}
Let Algorithm~\ref{alg.ParamRec} be called with a quasi-ordinary polynomial $f \in \E_0\ll\xu^{\Gamma_0}\rr[z]$ and $\mfn_0 \in \Q^n$ and write $(\sigma_i, \alpha_i) \in \hom(\Gamma_0,\E_i^*) \times \E_i\ll\xu^{\Gamma_i}\rr$ for the computed parametrizations. Then each root $\beta \in \overline{\E_0}\ll\xu^{\Q^n}\rr$ of $f$ with $\ord_{\cle}(\beta) \cgt \mfn_0$ is associated to at least one $(\sigma_i, \alpha_i)$.
\end{lem}
\begin{proof}
Let $\beta$ be a root of $f$ of order $\mfn \cgt \mfn_0$. If $\beta = 0$ or if $\beta \neq 0$ and $\edge_{\cle,\mfn}(f)$ is a linear equation then we consider the parametrization constructed in line $3$: The involved series is already $\beta$ and the involved automorphism is the identity.

Otherwise assume the outer loop is processing slope $\mfn$, set $\Gamma' := \Gamma_0 + \Z \mfn$ and write $\beta = \beta' + \beta''$ where $\beta' := \lt_{\cle}(\beta)$. Then $\beta'$ must be a root of $\edge_{\cle,\mfn}(f)$. Setting $\beta' = o \xu^{\mfn}$ for some $o \in \overline{\E_0}$ we have that $r := o^b$ is a root of $\sum_{0 \le j \le k} f_j z^j$, compare line $8$.

Assume now we are in the inner loop processing this $r \in \E'^*$ and let $f' := \sigma'\lift{z}(f)(z + \alpha')$ be the polynomial in the arguments to the recursive call in line $11$. Let $\tau' \in \hom(\Gamma',\overline{\E_0})$ be defined via $\tau' \vert_{\Gamma_0} = \sigma'^{-1}$ and $\mfn \mapsto r^{-u}o$. We have to check that this is well-defined, more precisely, since $b \mfn \in \Gamma_0$ (with $b$ minimal) we must have that $(r^{-u}o)^b$ coincides with $\sigma'^{-1}(b \mfn)$; Indeed \[(r^{-u}o)^b = r^{1-bu} = r^{\sum_{1 \le i \le n} v_i c_i} = \textstyle \prod_{1 \le i\le n} (r^{v_i})^{c_i} = \sigma'^{-1}(\textstyle\sum_{1 \le i \le n}c_i \mfm_i) = \sigma'^{-1}(b \mfn).\]
% In particular $\tau'(\alpha') = \tau'(r^u \xu^{\mfn}) = o \xu^{\mfn} = \beta'$ and $\tau'\lift{z}(f') = (\tau' \sigma')\lift{z}(f)(z + \tau'(\alpha')) = f(z + \beta')$. So $\beta''$ is a root of $\tau'\lift{z}(f')$ and, hence, $\tau'^{-1}(\beta'')$ is a root of $f'$ of order greater $\mfn$.
In particular $\tau'(\alpha') = \tau'(r^u \xu^{\mfn}) = o \xu^{\mfn} = \beta'$ and applying $\tau'^{-1}$ to $0 = f(\beta)$ we get $0 = (\tau'^{-1})\lift{z}(f)(\tau'^{-1}(\beta') + \tau'^{-1}(\beta'')) = (\sigma')\lift{z}(f)(\alpha' + \tau'^{-1}(\beta'')) = f'(\tau'^{-1}(\beta''))$. So $\tau'^{-1}(\beta'')$ is a root of $f'$ of order greater $\mfn$.

Assuming that the statement holds for the recursive call, we get a parametrization $(\sigma'', \alpha'') \in \hom(\Gamma_0,\E''^*) \times \E''\ll\xu^{\Gamma''}\rr$ which is associated to $\tau'^{-1}(\beta'')$ via, say, $\tau''$ with $\tau'' \vert_{\Gamma'} = \sigma''^{-1}$. This is combined to a returned parametrization $(\sigma''\sigma', \sigma''(\alpha') + \alpha'')$. We claim that this parametrization is associated to $\beta$ via $\tau'\tau''$. Indeed $\tau'\tau''$ restricts to the inverse of $\sigma''\sigma'$: \[(\tau'\tau'')\vert_{\Gamma_0} = (\tau'\vert_{\Gamma_0})(\tau''\vert_{\Gamma_0}) = \sigma'^{-1}(\sigma''^{-1}\vert_{\Gamma_0}) = (\sigma''\sigma')^{-1}\] It remains to show that it maps $\sigma''(\alpha') + \alpha''$ to $\beta$:
\begin{multline*}
\tau'\tau''(\sigma''(\alpha') + \alpha'') = \tau'(\tau''\sigma'')(\alpha') + \tau'(\tau''(\alpha'')) =\\ \tau'(\sigma''^{-1}\sigma'')(\alpha') + \tau'(\tau'^{-1}(\beta'')) = \tau'(\alpha') + \beta'' = \beta' + \beta'' = \beta\qedhere
\end{multline*}
\end{proof}

Now if $\beta_1$ and $\beta_2$ are roots of $f$ associated to $(\sigma_i,\alpha_i)$ via $\tau_1$ and $\tau_2$, then $\tau_1 \tau_2^{-1}$ restricts to the identity in $\hom(\Gamma_0, \overline{\E_0})$, hence, maps $\Gamma_0$ to $1$. In other words $\tau_1 \tau_2^{-1}$ acts as an automorphism on $\overline{\E_0}\ll\xu^{\Gamma_i}\rr$ fixing $\overline{\E_0}\ll\xu^{\Gamma_0}\rr$. Therefore, the roots which are associated to one and the same parametrization are all conjugate and, hence, roots of the same irreducible factor. This gives an injection from the irreducible factors (with roots of order greater $\mfn_0$) to parametrizations (of order greater $\mfn_0$).

\begin{cor}[Complete Sets of Rational Parametrizations]\label{cor.RatPar}
Let Algorithm~\ref{alg.Param} be called with a quasi-ordinary polynomial $f \in \E_0\ll\xu^{\Gamma_0}\rr[z]$ and write $(\sigma_i, \alpha_i) \in \hom(\Gamma_0,\E_i^*) \times \E_i\ll\xu^{\Gamma_i}\rr$ for the computed parametrizations, then:
\begin{itemize}
\item $\deg_z(f) = \sum_{i} [\E_i : \E_0]~\#(\Gamma_i / \Gamma_0)$
\item All computed parametrizations are rational and in bijective correspondence with the irreducible factors of $f$.
\end{itemize}
\end{cor}
\begin{proof}
To show this, we apply the above two lemmata with $\mfn_0 := (-1,\dots,-1)$ each. In Lemma~\ref{lem.BoundExts}, clearly, $\deg_z(f)$ is an upper bound for the righthand side of the inequality, so $\deg_z(f) \ge \sum_{i} [\E_i : \E_0]~\#(\Gamma_i / \Gamma_0)$.

Now let $f_j \vert f$ be the irreducible factors. By Lemma~\ref{lem.Completeness} and the previous discussion we may assume that there is an injection $j \mapsto i_j$ s.t.\ $f_j$ is parametrized by $(\sigma_{i_j}, \alpha_{i_j})$. Then we have chains of field inclusions \[\totfrac(\E_0\ll\xu^{\Gamma_0}\rr) \subseteq \totfrac(\E_0\ll\xu^{\Gamma_0}\rr[z]/\langle f_j \rangle) \subseteq \totfrac(\E_{i_j}\ll\xu^{\Gamma_{i_j}}\rr),\] where the second inclusion is given by $(\sigma_{i_j}, \alpha_{i_j})$. The extensions are algebraic of degrees $[\totfrac(\E_0\ll\xu^{\Gamma_0}\rr[z]/\langle f_j \rangle):\totfrac(\E_0\ll\xu^{\Gamma_0}\rr)]=\deg_z(f_j)$ respectively $[\totfrac(\E_{i_j}\ll\xu^{\Gamma_{i_j}}\rr):\totfrac(\E_0\ll\xu^{\Gamma_0}\rr)] = [\E_{i_j} : \E_0]~\#(\Gamma_{i_j} / \Gamma_0)$ and therefore \[\deg_z(f) = \sum_{j} \deg_z(f_j) \le \sum_{j} [\E_{i_j} : \E_0]~\#(\Gamma_{i_j} / \Gamma_0) \le \sum_{i} [\E_{i} : \E_0]~\#(\Gamma_{i} / \Gamma_0).\]

Together we have proven equality. In fact, since all summands are positive we find that the map $j \mapsto i_j$ is a bijection and for each chain of inclusions as above equality of degrees must hold. This gives the second statement.
\end{proof}

%%% Local Variables: 
%%% TeX-master: "~/doc/work/academic/worked_texts/Formal-Desing/formal-desings.tex"
%%% End: 

%% file: conclusion.tex
We have introduced the concept of formal desingularizations and shown how to compute them for hypersurfaces of $\P_{\E}^3$. The algorithm has been implemented and found to run very well. We can compute formal desingularizations faster by magnitudes than the general algorithms for desingularization. The reason is that our algorithm doesn't depend on Groebner basis computations. It relies only on linear algebra and polynomial factorization.

Here, we would like to point out an application of our algorithm. Namely, formal desingularizations can be used to compute adjoint sheaves. We will describe how to do that in another paper. Adjoint sheaves in arbitrary dimension are defined to be the direct image of the tensor powers of the canonic sheaf w.r.t.\ an arbitrary desingularization. Adjoint spaces for projective schemes ({\it i.e.}, graded components of the associated graded ring to that sheaf) are the keystone to the rational parametrization of curves. In \cite{JS:1998b} it has been shown that they are of equal importance for the computation of rational surface parametrizations. For example, they facilitate the computation of the arithmetic genus and the plurigenera of the surface. Thus, we can effectively check Castelnuovo's Criterion for the parametrizability of surfaces. Moreover, adjoint spaces can be used to construct certain rational maps that reduce the parametrization problem to a set of base cases. The final goal is an efficient implementation to rationally parametrize hypersurface in $\P_{\Q}^3$ (with or without introducing field extensions).

We finish with an open problem. With our current definition, a formal desingularization is just a loosely related set of formal prime divisors. For some applications, however, it would be nice to know the dual graph of the surface desingularization, {\it i.e.}, an annotated graph with one vertex for each exceptional divisor and edges whenever two divisors intersect. Such a graph could be the starting point for an algorithm to compute the minimal formal desingularization. In our method it would be easy to keep track of the dual graph of the embedded curve desingularization and in a certain sense the dual graph of the resolution projects to it. Such graph coverings have been studied in \cite{AN:2000, AN_AS:2000}. The problem is that we get formal prime divisors from two sources, namely, Algorithm~\ref{alg.DivisorsAboveCurve} and Algorithm~\ref{alg.DivisorsAboveCrossing}. However, it is not clear how the corresponding prime divisors intersect.

The algorithm of this paper as well as the method for the computation of adjoint spaces is available as a {\tt Magma}-package and can be downloaded via this link:\\
\centerline{\texttt{http://www.ricam.oeaw.ac.at/software/symcomp/adjoints.tar.gz}}
It will probably become part of the next major {\tt Magma}-release.

\section*{Acknowledgment}

I want to thank G{\'a}bor Bodn{\'a}r, Gavin Brown and my thesis advisor Josef Schicho for bringing up the subject. I am also grateful to Josef for many helpful suggestions.

%%% Local Variables: 
%%% TeX-master: "~/doc/work/academic/worked_texts/Formal-Desing/formal-desings.tex"
%%% End: 

%% file: local.tex
In this appendix we gather a few results from local commutative algebra and present them in a form suitable for our needs. Again $\E$ denotes a field of characteristic zero which need not be algebraically closed. We are dealing with completions of stalks of regular schemes of finite type over $\E$. Therefore we first give a famous structure theorem in this setting. Recall that an \emph{essentially finite} local $\E$-algebra is the localization of a finitely generated $\E$-algebra at a prime ideal.

\begin{thm}[Cohen Structure Theorem]\label{thm.Cohen}
Let $(\widehat{A},\mathfrak{m})$ be the completion of an essentially finite, regular, local $\E$-algebra of Krull dimension $s$. Set $\F_0 := \widehat{A}/\mathfrak{m}$ with canonic projection $\pi: \widehat{A} \to \F_0$. Further let
\begin{itemize}
\item $\{u_i\}_{1 \le i \le r} \subset \widehat{A}$ be a set projecting to a transcendence basis of $\F_0$ over $\E$ and
\item
$\{v_j\}_{1 \le j \le s} \subset \mathfrak{m}$ a minimal set of generators.
\end{itemize}
Then there is a unique coefficient field $\F \subseteq \widehat{A}$ containing $\E$ and $\{u_i\}_{1 \le i \le r}$ s.t.\ $\pi$ restricts to an isomorphism $\F \to \F_0$ and
\begin{gather*}
\F\ll x_1,\dots,x_s\rr  \to \widehat{A}:
\begin{cases}
f \mapsto f & \text{for $f \in \F$},\\
x_j \mapsto v_j & \text{for $1 \le j \le s$}
\end{cases}
\end{gather*}
is also an isomorphism.
\end{thm}
\begin{proof}
The existence of a unique field $\F$ fulfilling the first assertion is the content of \cite[Thm.~28.3]{HM:1989} and its proof (where references to differential bases can be substituted by transcendence bases in characteristic zero). Since $\{v_j\}_{1 \le j \le s}$ is a set of generators for $\mathfrak{m}$ one easily sees that the homomorphism in the second assertion is surjective, {\it cf.} \cite[Lem.~10.23]{IM_MA:1969}. The rings on both sides have the same dimension. Therefore the kernel must be trivial by \cite[Cor.~11.18]{IM_MA:1969}.
\end{proof}

We actually need a constructive version of a kind of inverse of the above isomorphism in a special case.

\begin{cor}[Completion at Points in Affine $n$-Space]\label{cor.PtCompl}
Let $\mathfrak{p} := \langle f_1, \dots, f_r \rangle \subset \E[x_1,\dots,x_n]$ be a prime ideal of height $l \le  \min(r,n)$,
\begin{gather*}
\mathcal{J} := \frac{\partial (f_1, \dots, f_r)}{\partial (x_1, \dots, x_n)} \in \E[x_1,\dots,x_n]^{r \times n}
\end{gather*}
the Jacobian matrix, $\F_0 := \totfrac(\E[x_1,\dots,x_n]/\mathfrak{p})$ the residue field, $\pi: \E[x_1,\dots,x_n] \to \F_0$ the canonic projection and
\begin{gather*}
%\phi: \E[x_1,\dots,x_n] \to \F_0[t_1,\dots,t_l]: x_i \mapsto \pi(x_i) + \langle M_i, (t_1,\dots,t_l)^T \rangle
\phi: \E[x_1,\dots,x_n] \to \F_0[t_1,\dots,t_l]: x_i \mapsto \pi(x_i) + M_i\,(t_1,\dots,t_l)^T,
\end{gather*}
where $M_i$ are row vectors with entries in $\{0,1\}$. Write $M := (M_i)_{1 \le i \le n}$.

Then $\phi$ extends uniquely to a homomorphism $\widehat{\E[x_1,\dots,x_n]_{\mathfrak{p}}} \to \F_0\ll t_1,\dots,t_l\rr$. Moreover we can choose $M$ s.t.\ $\pi(\mathcal{J})\,M \in \F_0^{r \times l}$ has rank $l$ and then the extended homomorphism becomes an isomorphism.
\end{cor}
\begin{proof}
The homomorphism $\phi$ trivially extends to $\F_0\ll t_1,\dots,t_l\rr$. Then for any $g \in \E[x_1,\dots,x_n]$ one computes for the image
\begin{gather*}
\phi(g) \in \pi(g) + \pi(\partial g/\partial (x_1, \dots, x_n))M(t_1,\dots,t_l)^T + \mathfrak{m}^2
\end{gather*}
where $\mathfrak{m} := \langle t_1, \dots, t_l \rangle$. Now the above equation shows that if $g \in \E[x_1,\dots,x_n] \setminus \mathfrak{p}$ then we have $\pi(g) \neq 0$ for the constant part. Therefore the image is a unit and the homomorphism lifts uniquely to $\E[x_1,\dots,x_n]_{\mathfrak{p}}$. Impose the $\mathfrak{p}$-adic topology on the domain and the $\mathfrak{m}$-adic topology on the codomain. Writing down the long expansion one sees that the homomorphism is even uniformly continuous. Since the codomain is already complete we have a unique lifting to $\widehat{\E[x_1,\dots,x_n]_{\mathfrak{p}}}$.

Since $\E[x_1,\dots,x_n]$ is regular the Jacobian image $\pi(\mathcal{J})$ has rank $l$ equal to the height of the prime. Let $(j_1,\dots,j_l)$ be the column indices of a non-vanishing $l \times l$-minor and choose $M := (e_{j_1}, \dots, e_{j_l})$. Here $e_j$ denotes the column vector with $1$ in position $j$ and $0$ otherwise. Since multiplication by $M$ extracts exactly the columns of this minor also $\pi(\mathcal{J})M$ has rank $l$.

Applying the above formula to the components of the vector $(f_1,\dots,f_r)^T$ of generators of $\mathfrak{p}$ we find that its $\pi$-image is of the form \[\underbrace{(\pi(f_1),\dots,\pi(f_r))^T}_{=(0,\dots,0)} + \underbrace{\pi(\mathcal{J})M(t_1,\dots,t_l)^T}_{=: L(f_1,\dots,f_r)} + \text{ (higher order terms).}\] The rank condition assures that the components of $L(f_1,\dots,f_r)$ generate $\mathfrak{m}/\mathfrak{m}^2$ as $\F_0$-vector space. Hence $\mathfrak{m} = \langle \phi(f_1), \dots, \phi(f_r) \rangle$ by Nakayama's lemma.

Now let $\F \subseteq \widehat{\E[x_1,\dots,x_n]_{\mathfrak{p}}}$ be a coefficient field in the sense of Theorem~\ref{thm.Cohen}., {\it i.e.}, restricting the canonic projection to $\F$ gives an isomorphism \[\F \cong \widehat{\E[x_1,\dots,x_n]_{\mathfrak{p}}}/\mathfrak{p}\widehat{\E[x_1,\dots,x_n]_{\mathfrak{p}}} \cong \F_0.\] But then also $\phi(\F)$ is a coefficient field of $\F_0\ll t_1,\dots,t_l\rr$. Now we proceed as in the proof of Theorem~\ref{thm.Cohen}. The paragraph above shows that $\phi$ is surjective and comparing dimensions one proves injectivity.
\end{proof}

To show correctness of our algorithms we need that completion commutes with two common operations, namely, building the integral closure and computing the blow up algebra.

\begin{lem}[Integral Closure and Completion]\label{lem.NormCompl}
Let $A$ be a finitely generated $\E$-algebra, $f \in A[z]$ a monic polynomial, $\mathfrak{p} \subset A$ a prime ideal and consider the following diagram:\\
\centerline{\xymatrix{
\intclos{A[z]/\langle f \rangle} \ar[r] & \intclos{\widehat{A_{\mathfrak{p}}}[z]/\langle f \rangle} \\
A \ar[r] \ar[u] & \widehat{A_{\mathfrak{p}}} \ar[u]
}}

Then $\mathfrak{q} \mapsto \mathfrak{q}' := \mathfrak{q} \intclos{\widehat{A_{\mathfrak{p}}}[z]/\langle f \rangle}$ gives a bijective correspondence between prime ideals $\mathfrak{q} \subset \intclos{A[z]/\langle f \rangle}$ above $\mathfrak{p}$ and prime ideals $\mathfrak{q}' \subset \intclos{\widehat{A_{\mathfrak{p}}}[z]/\langle f \rangle}$ above $\mathfrak{p}\widehat{A_{\mathfrak{p}}}$, and the induced homomorphisms $\widehat{\intclos{A[z]/\langle f \rangle}_{\mathfrak{q}}} \to \widehat{\intclos{\widehat{A_{\mathfrak{p}}}[z]/\langle f \rangle}_{\mathfrak{q}'}}$ are isomorphisms.
\end{lem}
\begin{proof}
The fact that building the integral closure and completion commutes is known as Zariski's Main Theorem, see, {\it e.g.}, \cite[Thm.\ VIII.33]{OZ_PS:1960}.
\end{proof}

\begin{lem}[Blowing up and Completion]\label{lem.BlowCompl}
Let $A$ be a finitely generated $\E$-algebra, $\mathfrak{p} \subset A$ a prime ideal and consider the diagram\\
\centerline{\xymatrix{
B \ar[r] & B' := \widehat{A_{\mathfrak{p}}} \otimes_A B \\
A \ar[r] \ar[u] & \widehat{A_{\mathfrak{p}}} \ar[u]
}}
where $B$ is a coordinate ring of an affine chart of the blow up of $\spec A$ at $\mathfrak{p}$.

Then $B'$ is a coordinate ring of an affine chart of the blow up of $\spec \widehat{A_{\mathfrak{p}}}$ at $\mathfrak{m}\widehat{A_{\mathfrak{p}}}$. Further $\mathfrak{q} \mapsto \mathfrak{q}' := \mathfrak{q} B'$ gives a bijective correspondence between prime ideals $\mathfrak{q} \subset B$ above $\mathfrak{p}$ and prime ideals $\mathfrak{q}' \subset B'$ above $\mathfrak{m}\widehat{A_{\mathfrak{p}}}$, and the induced homomorphisms $\widehat{B_{\mathfrak{q}}} \to \widehat{B'_{\mathfrak{q}'}}$ are isomorphisms.
\end{lem}
\begin{proof}
This is an algebraic transcription of \cite[Prop.\ A.14.7]{KK_JLV:2004b} which holds analogously for ground fields which are not algebraically closed.
\end{proof}

Finally we state the theorem which provides the theoretical basis for the whole formal desingularization procedure.

\begin{thm}[Theorem of Jung-Abhyankar]\label{thm.JungAbhyank}
Let $f \in \E\ll x_1,\dots,x_n\rr [z]$ be a \emph{monic}, \emph{squarefree} polynomial s.t.\ $\disc_z(f) = x_1^{e_1} \cdots x_m^{e_m} u(x_1,\dots,x_n)$ where $m \le n$ and $u(0,\dots,0) \neq 0$. Then there is a natural number $d \ge 1$ and there are $\deg_z(f)$ distinct power series $\alpha_i \in \overline{\E}\ll x_1^{1/d},\dots,x_m^{1/d},x_{m+1},\dots,x_{n}\rr$ solving $f$, {\it i.e.}, $f(x_1, \dots, x_n, \alpha_i)=0$ for $1 \le i \le \deg_z(f)$.
\end{thm}
\begin{proof}
An irreducible factor of $f$ must again be monic and its discriminant must be a factor of $\disc_z(f)$, hence, can also be written as above. Now the statement can be found in \cite[Prop.\ 3.2.5]{KK_JLV:2004b}.
\end{proof}

%%% Local Variables: 
%%% TeX-master: "~/doc/work/academic/worked_texts/Formal-Desing/formal-desings.tex"
%%% End: 